\documentclass[12pt,reqno]{amsart}
\usepackage{amsthm,amsfonts,amssymb,euscript}

\newtheorem{theorem}{Theorem}[section]
\newtheorem{lemma}[theorem]{Lemma}
\newtheorem{proposition}[theorem]{Proposition}
\newtheorem{corollary}[theorem]{Corollary}
\newtheorem{definition}[theorem]{Definition}
\newtheorem{remark}[theorem]{Remark}

\numberwithin{equation}{section}

\begin{document}

\title{Global regularity for the energy-critical NLS on $\mathbb{S}^3$.}

\author{Benoit Pausader}
\address{Universit\'e Paris 13, Sorbonne Paris Cit\'e, LAGA, CNRS, (UMR 7539), F-93430, Villetaneuse, France. }
\email{pausader@math.u-paris13.fr}

\author{Nikolay Tzvetkov}
\address{Universit\'e Cergy-Pontoise, UMR CNRS 8088, Cergy-Pontoise 95000, France }
\email{nikolay.tzvetkov@u-cergy.fr}

\author{Xuecheng Wang}
\address{Princeton University}
\email{xuecheng@math.princeton.edu}

\thanks{B.P. was partially supported by NSF grant DMS-1142293. N.T. is partially supported by the ERC grant Dispeq.}

\begin{abstract}
We establish global existence for the energy-critical nonlinear Schr\"odinger equation on $\mathbb{S}^3$. This follows similar lines to the work on $\mathbb{T}^3$ but requires new extinction results for linear solutions and
bounds on the interaction of a Euclidean profile and a linear wave of much higher frequency that are adapted to the new geometry.
\end{abstract}

\maketitle

\section{Introduction}

We consider the question of global well-posedness for the defocusing energy-critical nonlinear Schr\"odinger equation on $\mathbb{S}^3$, namely
\begin{equation}\label{NLS}
\left(i\partial_t-\Delta_{\mathbb{S}^3}\right)u+\vert u\vert^4u=0.
\end{equation}

\medskip

The goal of this work is to apply the method introduced by Ionescu and the first author in \cite{IoPa, IoPa2} to the energy critical NLS on the three dimensional sphere. We therefore follow the same general lines. The main novelty in this paper is the proof of the extinction lemma for the linear flow and the bound on the interaction between a high-frequency linear wave and a low frequency profile which in the case of the sphere requires new arguments related to the different geometry.

\medskip

The study of the Schr\"odinger equation on compact manifolds was initiated by Bourgain \cite{Bo2,Bo1} for torii and systematically developed by Burq-G\'erard-Tzvetkov for arbitrary compact manifolds, where the sphere appeared as a natural challenging problem, somewhat complementary to the case of the torus. More precisely, on the torus, the spectrum is badly localized, but still regular and with low multiplicity and there is a nice basis of eigenfunctions coming from the product structure; on the sphere, the spectrum is as simple as it can be, but has very high multiplicity, with eigenfunctions of different character which are in some sense as bad as can be. Informally speaking, on the sphere, the oscillations in time and in space appear as rather decoupled and have to be treated differently. We also refer to \cite{Ba,BaCaDu,BaIg,Bo2,GePi,Ha,HaPa,HeTaTz2,IoPaSt,KiViZh,TaTz,TzVi} for other works on the nonlinear Schr\"odinger equation in different geometries.

On the torus $\mathbb{T}^3$, Bourgain \cite{Bo2} proved global existence for subquintic nonlinearities. Local existence for the energy-critical problem was obtained in \cite{HeTaTz} and extended to global existence in \cite{IoPa2}.
Global existence for the defocusing problem on $\mathbb{S}^3$ was obtained for subquintic nonlinearities in \cite{BuGeTz}, local existence for the quintic problem was established in \cite{He}. In this paper, we prove global existence for the energy-critical problem, namely
\begin{theorem}\label{MainThm}
For any $u_0\in H^1(\mathbb{S}^3)$, there exists a unique global strong solution of \eqref{NLS} satisfying $u(0)=u_0$. In addition, if $u_0\in H^s$ for some $s\ge 1$, then $u\in C(\mathbb{R}:H^s)$.
\end{theorem}
Since the Cauchy problem is ill-posed in $H^1$ for superquintic nonlinearities (see \cite{BuGeTz3}), this completes the local and global analysis of wellposedness in $H^1$. With the results in \cite{CKSTTcrit,IoPaSt}, this establishes global existence for the energy-critical problem in $\mathbb{R}^3$, $\mathbb{H}^3$ and $\mathbb{S}^3$.

For supercritical nonlinearities, classical compactness results yield global existence of weak solutions for \eqref{NLS}, see e.g. \cite{Cazenave:book}. Their uniqueness (and regularity) is, however an open problem. In particular, the results in \cite{AlCa} suggest the possibility of $H^s$ loss of regularity for weak solutions for some $1\le s\le 3/2$.

The proof of Theorem 1.1 brings together the different contributions developed in \cite{BuGeTz,BuGeTz3,BuGeTz2,BuGeTz4,He,HeTaTz,IoPa2} which address (among other things) the subcritical nonlinear Schr\"odinger equation, the analysis of products of eigenfunctions, boundedness of the first iterate in the energy-critical case, global existence for large data for the energy-critical problem and global analysis of the corresponding problem in the case of the torus $\mathbb{T}^3$.

\medskip

In the study of the nonlinear Schr\"odinger equation on a manifold, the (difficult) study of the linear flow is very important and is presumably specific to each particular setting. This is one of the major ingredients that limit the generality of the present work and we do not add in new information on that aspect (and we do rely heavily on the analysis developed in \cite{Bo1,BuGeTz3,He,HeTaTz}). A ``good'' understanding of the linear flow should automatically yield global existence for the defocusing energy-subcritical problem, and local existence and stability for the energy-critical problem.

In addressing the global existence for the energy-critical problem on the sphere, we need to revisit the main nonlinear ingredients in \cite{IoPa2} and reinterpret them. While we are not yet able to give a general result, even conditionally on a good linear theory, several aspects start to emerge for the key ingredients.

The first one concerns the application of the profile decomposition, which seems to hold in a very general context. To properly work, it requires an extinction argument which is provided here by Lemma \ref{Extinction}. Since one already has sufficiently good Strichartz estimates, one only needs an improvement on the Sobolev inequality for the linear flow. This comes from two aspects. On the one hand, by purely elliptic considerations, one can track down when the Sobolev inequality is inefficient. Very precise estimate are available to quantify this (see e.g. \cite{So1,So2}). Here, since we need to beat this inequality by a fixed but large constant, we rely on the explicit formula for the eigenprojectors, but in general, such information might follow from estimate of the Green function away from the diagonal.

Once this has been taken into account, we are left with a part of the solution that has more structure and we need to use the fact that, under the linear flow, it cannot remain concentrated for all times, which, for the moment, we can only do using some argument coming from the Euclidean Fourier transform, or from Weyl bounds, which are quite sensitive to fine properties of the spectrum. This is done here in Lemma \ref{BouLem}.

The second main ingredient is an understanding of the linearization of the equation around an arbitrary profile for certain initial data (the remainder in the profile decomposition). In general, we expect solutions to essentially follow the linear flow. In the Euclidean $\mathbb{R}^3$ case, this would follow from local smoothing estimates. In a compact manifold, this might follow for short time if one can get quantitative bounds on the concentration of eigenfunctions of $\Delta_{\mathbb{S}^3}$ to points, i.e. the absence of semi-classical measure concentrating on points (see e.g. \cite{AnMa}). Such information is provided by Lemma \ref{LocProj} which is valid for an arbitrary smooth manifold. This is then used in Lemma \ref{HFLFI} to control the first iterate of the above mentioned linearization, but this latter result uses the particular localization of the spectrum on the sphere in an essential way (see also \cite{IoPa2} for a similar arguments relying more on the ``Euclidean-like'' localization of the spectrum - in the sense that it forms a $3$-dimensional lattice).

\medskip

While the analysis in \cite{IoPa,IoPa2} can probably be combined with the new estimate on the linear flow in \cite{Bo4} to yield global existence for the defocusing energy-critical Schr\"odinger equation on $\mathbb{T}^4$, let us mention several other open problems with increasing (in our opinion) level of difficulty. 1) The analysis developed here might extend to the case of Zoll manifolds provided one obtains the appropriate bounds on eigenprojectors, possibly from arguments in the spirit of Lemma \ref{LemBuGeTz}. 2) The analysis of the same problem in the space $\mathbb{S}^2\times\mathbb{S}^1$ seems to require nontrivial adaptations from the arguments given in \cite{He,IoPa2} and here, even for small initial data. This is partly due to the failure of good $L^4$ bilinear estimates for eigenprojectors. 3) The case of $\mathbb{S}^4$ remains a challenging open problem where new ideas seem needed due to the failure of the $L^4_{x,t}$-Strichartz estimates which implies that the second iterate is unbounded, see \cite{BuGeTz}.

\medskip

Another interesting case that can be addressed with a similar analysis is the energy-critical problem in the unit ball $B(0,1)\subset\mathbb{R}^3$ with Dirichlet boundary condition and {\it radial} data\footnote{The case of arbitrary data remains an outstanding open problem, where even the linear flow is still not satisfactorily understood \cite{An,An2,BlSmSo,BlSmSo2,Pla}.}. In Section \ref{SecBall}, we shall give the main modifications required to prove

\begin{theorem}\label{BallThm}
Let $s\ge 1$. For any $u_0\in H^s\cap H^1_D(B(0,1))$ radial\footnote{Here $H^1_D$ is the completion for the $H^1$-norm of the smooth functions compactly supported in $B(0,1)$.}, there exists a unique strong solution of \eqref{NLS}, $u\in C(\mathbb{R}:H^s) $.
\end{theorem}

Global existence for finite-energy solutions to the nonlinear Schr\"odinger equation on two dimensional domains was already obtained by Anton \cite{An}. We also refer to \cite{An2,BlSmSo,BlSmSo2,Pla} for other results in three dimensions and to \cite{Dod,IvPl,KiViZh,LiSmZh} for global existence and scattering results in the exterior of the unit ball.

In Section \ref{SecNot}, we review some notation and introduce our main spaces. In Section \ref{localwp}, we review the local well-posedness theory. In Section \ref{SecProf}, we present the profile decomposition on $\mathbb{S}^3$. In Section \ref{SecGWP}, we prove Theorem \ref{MainThm}. Finally in Section \ref{App}, we prove some additional results needed in the course of the proof and  give the ingredients for the proof of Theorem \ref{BallThm}.


\section{Notations and preliminaries}\label{SecNot}

In this section we summarize our notations and collect several lemmas that are used in the rest of the paper.

Given two quantities $A$ and $B$, the notation $A\lesssim B$ means that $A\le C B$, with $C$ uniform with respect to the set where $A$ and $B$ varies. We write $A\simeq B$ when $A\lesssim B\lesssim A$. If the constant $C$ involved has some explicit dependency, we emphasize it by a subscript. Thus $A\lesssim_uB$ means that $A\le C(u)B$ for some constant $C(u)$ depending on $u$.

We write $F(z)=z\vert z\vert^4$ the nonlinearity in \eqref{NLS}. For $p\in\mathbb{N}^n$ a vector, we denote by $\mathfrak{O}_{p_1,\dots ,p_n}(a_1,\dots,a_n)$ a $\vert p\vert$-linear expression which is a product of $p_1$ terms which are either equal to $a_1$ or its complex conjugate $\overline{a}_1$ and similarly for $p_j$, $a_j$, $2\le j\le n$.

\subsection{The three sphere}

We can view $\mathbb{S}^3$ as the unit sphere in the quaternion field and this endows $\mathbb{S}^3$ with a group structure with the north pole $O=(1,0,0,0)$ as the unit element. This also endows $\mathbb{S}^3\subset\mathbb{R}^4$ with the structure of a riemmannian manifold with distance $d_g$ which is also given by
\begin{equation*}
d_g(P,Q)=\angle(P,Q),
\end{equation*}
where $\angle(P,Q)$ denotes the angle between the rays starting at the origin and passing through $P$ and $Q$. For $Q\in\mathbb{S}^3$, we define $R_Q$ to be the right multiplication by $Q^{-1}$. This defines an isometry of $\mathbb{S}^3$.

We can parameterize $\mathbb{S}^3$ in exponential radial coordinates $P\mapsto (\theta,\omega)$ where $\theta=d_g(O,P)$ and $\omega\in\mathbb{S}^2$. In fact we have the global mapping\footnote{Here by $\mathbb{S}^1$ we mean $[0,2\pi]$ with the endpoints identified.}
\begin{equation*}
[0,\pi]\times[0,\pi]\times\mathbb{S}^1\ni(\theta,\psi,\varphi)\mapsto(\cos\theta,\sin\theta\cos\psi,\sin\theta\sin\psi\cos\varphi,\sin\theta\sin\psi\sin\varphi).
\end{equation*}
In these coordinates, we have that
\begin{equation}\label{Delta}
\begin{split}
\Delta_{\mathbb{S}^3}&=\frac{1}{\sin^2\theta}\frac{\partial}{\partial\theta}\sin^2\theta\frac{\partial}{\partial\theta}+\frac{1}{\sin^2\theta}\Delta_{\mathbb{S}^2}=\frac{\partial^2}{\partial\theta^2}+\frac{2\cos\theta}{\sin\theta}\frac{\partial}{\partial\theta}+\frac{1}{\sin^2\theta}\Delta_{\mathbb{S}^2}\\
&=\frac{1}{\sin^2\theta}\frac{\partial}{\partial\theta}\sin^2\theta\frac{\partial}{\partial\theta}+\frac{1}{\sin^2\theta\sin\psi}\frac{\partial}{\partial\psi}\sin\psi\frac{\partial}{\partial\psi}+\frac{1}{\sin^2\theta\sin^2\psi}\frac{\partial^2}{\partial\varphi^2}.
\end{split}
\end{equation}
In these coordinates, we also have the explicit formula for the Haar measure
\begin{equation*}
d\nu_g=(\sin\theta)^2\sin\psi d\theta d\psi d\varphi.
\end{equation*}

\subsection{Spherical harmonics}
We will consider the operator $L=-\Delta_{\mathbb{S}^3}+1$. For $k\in\mathbb{N}^\ast$, we define $\mathcal{E}_k$ to be the space of $k-1$-th spherical harmonics. We have an $L^2$-orthonormal decomposition
\begin{equation*}
L^2(\mathbb{S}^3)=\bigoplus_{k\in\mathbb{N}^\ast}\mathcal{E}_k
\end{equation*}
and $\pi_k$ defined above is the orthogonal projection on $\mathcal{E}_k$. These satisfy that for any $\varphi\in\mathcal{E}_k$, $L\varphi=k^2\varphi$. We recall the following bounds from Sogge \cite{So3}
\begin{equation}\label{SogBd}
\Vert \pi_qf\Vert_{L^p(\mathbb{S}^3)}\lesssim q^{1-3/p}\Vert f\Vert_{L^2(\mathbb{S}^3)},\,\, 4\le p\le \infty.
\end{equation}

We then define projectors on $I\subset\mathbb{R}$ by
\begin{equation}\label{DefPN}
P_I=\sum_{k\in I}\pi_k,\quad P_{\le N}=\sum_{k\in\mathbb{N}}\eta(\frac{k}{N})\pi_k,\quad P_N=P_{\le N}-P_{\le N/2}=\sum_{k\in\mathbb{N}}\eta_N(k)\pi_k.
\end{equation}
for $\eta\in C^\infty_c(\mathbb{R})$ such that $\eta(x)=1$ when $\vert x\vert\le 1$ and $\eta(x)=0$ when $\vert x\vert\ge 2$ and where $\eta_N(x)=\eta(x/N)-\eta(2x/N)$. In particular all the sums over $N$ below are implicitly taken to be over all dyadic integers, $N=2^k$ for some $k\in\mathbb{N}$.

In fact, we can be more precise about the spectral projectors. We define the Zonal function of order $k$, ${\bf Z}_k$ as
\begin{equation}\label{Zonal}
Z_k(\theta)=k\frac{\sin(k\theta)}{\sin\theta},\quad{\bf Z}_k( P)=Z_k(\angle(P,O)),
\end{equation}
where $O$ denotes the north pole. One may directly check that these are eigenfunctions of the Laplace-Beltrami operator on $\mathbb{S}^3$ defined in \eqref{Delta}. These allow to get the following classical result:
\begin{lemma}
The spectral projection on the $k-1$-th eigenspace can be written as
\begin{equation}\label{DefSpecProj}
\left[\pi_kf\right](P )=\frac{1}{2\pi^2}\int_{\mathbb{S}^3}{\bf Z}_k(R_PQ)f(Q)d\nu_g(Q).
\end{equation}
\end{lemma}

\begin{proof}
Denote, for this proof only $\Pi_k$ as the operator defined by the right-hand side of \eqref{DefSpecProj}. Using the symmetry
\begin{equation*}
{\bf Z}_k(R_PQ)=Z_k(\angle (P,Q))={\bf Z}_k(R_QP)=\overline{{\bf Z}_k(R_QP)},
\end{equation*}
and the fact that since $R_Q$ is an isometry, $\Pi_k$ commutes with $\Delta_{\mathbb{S}^3}$ and
we see that $L\Pi_kf=k^2\Pi_kf$. This also shows that $\Pi_k$ is self-adjoint. Therefore it is sufficient to prove that for any $g\in C^{10}(\mathbb{S}^3)$, there holds that
\begin{equation}\label{CompletenessOfPik}
g=\sum_{k\ge 1}\Pi_kg.
\end{equation}
Since $\Pi_k$ commutes with rotations, it suffices to prove that this equality holds at the north pole $O$. We switch to exponential coordinates. Using Fourier analysis on $[0,\pi]$, we see that
\begin{equation*}
\sin\theta\cdot g(\theta,\omega)=\sum_{k\ge 1}c_k(\omega)\sin(k\theta),\quad c_k(\omega)=\frac{2}{\pi}\int_0^\pi g(\theta,\omega)\sin(\theta)\sin(k\theta)d\theta.
\end{equation*}
In other words,
\begin{equation*}
g(\theta,\omega)=\sum_{k\ge 1}c_k(\omega)\frac{\sin(k\theta)}{\sin\theta}.
\end{equation*}
Integrating this over $\omega\in\mathbb{S}^2$ and letting $\theta\to 0$, we find (since $c_k(\omega)\in l^1_k(k^2)$ uniformly\footnote{Here we denote $l^1_k(k^2)$ the set of sequence which are summable in $k$ for the measure $k^2dk$. We also denote $d\omega$ the Haar measure on $\mathbb{S}^2$.} in $\omega$) that
\begin{equation*}
\begin{split}
g(O)&=\lim_{\theta\to0}\frac{1}{4\pi}\int_{\mathbb{S}^2}g(\theta,\omega)d\omega=\lim_{\theta\to0}\sum_{k\ge 1}\frac{1}{4\pi}\int_{\mathbb{S}^2}c_k(\omega)\frac{\sin(k\theta)}{\sin\theta}d\omega\\
&=\sum_{k\ge 1}\frac{1}{4\pi}\int_{\mathbb{S}^2}kc_k(\omega) d\omega=\sum_{k\ge 1}\frac{1}{2\pi^2}\int_0^\pi\int_{\mathbb{S}^2}g(\theta,\omega)k\frac{\sin (k\theta)}{\sin\theta} \sin^2\theta d\theta d\omega\\
&=\sum_{k\ge 1}\Pi_kg(O).
\end{split}
\end{equation*}
This shows \eqref{CompletenessOfPik} and finishes the proof.
\end{proof}

The spectral projectors $\pi_q$ satisfy a convenient reproducing formula highlighted in \cite{BuGeTz}: for $\chi\in\mathcal{S}(\mathbb{R})$ such that $\chi(0)=1$ and $\hat{\chi}$ supported on $[\varepsilon,2\varepsilon]$,
\begin{equation}\label{DefChi}
\chi_q\pi_q=\pi_q\chi_q=\pi_q,\quad\chi_q=\chi(\sqrt{L}-q).
\end{equation}
The interest of this comes from the following description of $\chi_q$:
\begin{lemma}[\cite{BuGeTz3}, Lemma 2.3.]\label{LemBuGeTz}

There exists $\varepsilon_0>0$ such that for every $\varepsilon\in (0,\varepsilon_0)$, we can decompose
\begin{equation}\label{DecChi}
\chi_q=qT_q+R_q,\quad\Vert R_q\Vert_{L^2\to H^{10}}\lesssim q^{-10}
\end{equation}
and there exists $\delta>0$ such that for any $x_0\in \mathbb{S}^3$, there exists a system of coordinates centered at $x_0$ such that for any $\vert x\vert\le\delta$,
\begin{equation*}
T_qf(x)=\int_{\mathbb{R}^3}e^{-iqd_g(x,y)}a(x,y,q)f(y)dy,
\end{equation*}
where $a(x,y,q)$ is a polynomial in $1/q$ with smooth coefficients supported on the set
\begin{equation*}
\{(x,y)\in V\times V:\vert x\vert\le\delta\ll\varepsilon/C\le\vert y\vert\le C\varepsilon\}.
\end{equation*}
\end{lemma}

In the study of the linearization of \eqref{NLSBis} at a profile, we will need the following quantitative version of the fact that quantum measures do not concentrate on points.

\begin{lemma}\label{LocProj}
Let $N\ge 1$ be a dyadic number and fix $P\in\mathbb{S}^3$, then there holds that
\begin{equation}\label{CrucialGain1}
\Vert \mathfrak{1}_{B(P,N^{-1})}\pi_q\Vert_{L^2\to L^2}=\Vert \pi_q\left[\mathfrak{1}_{B(P,N^{-1})}\cdot\right]\Vert_{L^2\to L^2}\lesssim N^{-1/2}+q^{-2}.
\end{equation}
\end{lemma}

\begin{remark}
Note that this estimate is sharp when testing against zonal harmonics of degree $p\ge N$. In addition, the proof holds on any compact smooth riemmannian manifold.
\end{remark}

\begin{proof}

This claim essentially follows from \cite{BuGeTz3}. We give here the modification necessary to obtain it. It suffices to prove the second bound as the first follows by duality. Also, we may assume that $N\gg1$.

Using \eqref{DefChi} and \eqref{DecChi}, remarking that
\begin{equation*}
\pi_q\mathfrak{1}_{B(P,N^{-1})}=q\pi_qT_q\mathfrak{1}_{B(P,N^{-1})}+\pi_qR_q\mathfrak{1}_{B(P,N^{-1})},\quad\Vert\pi_q\Vert_{L^2\to L^2}\le 1, 
\end{equation*}
we see that it suffices to show that
\begin{equation*}
\Vert T_q\mathfrak{1}_{B(P,N^{-1})}\Vert_{L^2\to L^2}\lesssim (q^2N)^{-1/2}.
\end{equation*}
Now, using the notation of \cite[page 12]{BuGeTz3}, we can decompose
\begin{equation*}
T_q=\int_{r=\delta_1}^{\delta_2} T^r_q dr
\end{equation*}
where for a finite number of charts covering $\mathbb{S}^3$ and centered at points $x_k$, there holds that
\begin{equation*}
\begin{split}
&\left[\mathfrak{1}_{B(x_k,\delta)}T_q^rf\right](Q)=\int_{\mathbb{S}^2}e^{-iqd_g(Q,\exp_{x_k}(r\omega))}a(Q,\exp_{x_k}(r\omega),q)\kappa(r,\omega)f_r(\omega)d\omega,\\
&f_r(\omega)=f(\exp_{x_k}(r\omega)),
\end{split}
\end{equation*}
where $\kappa$ is a new smooth function. Applying H\"older's inequality in $r$, we obtain for any $Q\in B(x_k,\delta)$
\begin{equation*}
\vert T_q(\mathfrak{1}_{B(P,N^{-1})}f)(Q)\vert^2\lesssim N^{-1}\int_{r=\delta_1}^{\delta_2} \vert T^r_qf(Q)\vert^2 dr
\end{equation*}
since by the triangle inequality, for any $x_k$, we have that
\begin{equation*}
Q\in B(P,N^{-1}),\, d_g(x_k,Q)=r\quad \Rightarrow\quad d_g(x_k,P)-N^{-1}\le r\le d_g(x_k,P)+N^{-1}.
\end{equation*}
The result then follows from \cite[Lemma 2.14]{BuGeTz3} which implies that
\begin{equation*}
q\Vert T^r_qf_r\Vert_{L^2}\lesssim \Vert f_r\Vert_{L^2}.
\end{equation*}
\end{proof}

\subsection{Linear analysis}

In fact, for simplicity of notations, we will replace equation \eqref{NLS} by
\begin{equation}\label{NLSBis}
\left(i\partial_t+L\right)u+\vert u\vert^4u=0.
\end{equation}
This is completely equivalent since a solution $u(x,t)$ solves \eqref{NLSBis} if and only if $v(x,t)=e^{-it}u(x,t)$ solves \eqref{NLS}.

\medskip

For solutions of \eqref{NLSBis}, we recall the conservation laws
\begin{equation}\label{conserve}
E(u)=\frac{1}{2}\int_{\mathbb{S}^3}\left[\vert\nabla u(x)\vert^2+\frac{1}{3}\vert u(x)\vert^6\right]dx,\quad M(u)=\int_{\mathbb{S}^3}\vert u(x)\vert^2dx.
\end{equation}
Here and below $dx$ refers to the Haar measure on $\mathbb{S}^3$.
These conserved quantities provide a uniform in time control on the $H^1$ norm and motivate our choice of function spaces.

\medskip

{\bf{Function spaces.}} The strong spaces are similar to the one used by Herr \cite{He}, adapting previous ideas from Herr-Tataru-Tzvetkov \cite{HeTaTz,HeTaTz2}. Namely
\begin{equation}\label{KoTa}
\begin{split}
\Vert u\Vert_{\widetilde{X}^s(\mathbb{R})}&:=\left(\sum_{k\in\mathbb{N}^\ast} k^{2s}\Vert e^{itk^2}\pi_ku(t)\Vert_{U^2_t(L^2)}^2\right)^\frac{1}{2},\\
\Vert u\Vert_{\widetilde{Y}^s(\mathbb{R})}&:=\left(\sum_{k\in\mathbb{N}^\ast} k^{2s}\Vert e^{itk^2}\pi_ku(t)\Vert_{V^2_t(L^2)}^2\right)^\frac{1}{2},
\end{split}
\end{equation}
where we refer to \cite{HaHeKo,He,HeTaTz,HeTaTz2,KoTa} for a description of the spaces $U^p(L^2),V^p(L^2)$ and of their properties. Note in particular that
\begin{equation*}
\widetilde{X}^1(\mathbb{R})\hookrightarrow \widetilde{Y}^1(\mathbb{R})\hookrightarrow L^\infty(\mathbb{R},H^1).
\end{equation*}
We denote by $U^p_L(L^2)$ the space $e^{itL}U^p(L^2)$.

For intervals $I\subset\mathbb{R}$, we define $X^s(I)$, $s\in\mathbb{R}$, in the usual way as restriction norms, thus
\begin{equation*}
X^1(I):=\{u\in C(I:H^1):\Vert u\Vert_{X^s(I)}:=\sup_{J\subseteq I,\,\vert J\vert\leq 1}[\inf_{v\cdot\mathbf{1}_J(t)=u\cdot\mathbf{1}_J(t)}\Vert v\Vert_{\widetilde{X}^s}]<\infty\}.
\end{equation*}
The spaces $Y^s(I)$ are defined in a similar way. The norm controlling the inhomogeneous term on an interval $I=(a,b)$ is then defined as
\begin{equation}\label{NNorm}
\Vert h\Vert_{N(I)}:=\Big \Vert\int_a^te^{i(t-s)L}h(s)ds\Big\Vert_{X^1(I)}.
\end{equation}

We also need a weaker critical norm
\begin{equation}\label{NewS319}
\begin{split}
&\Vert u\Vert_{Z(I)}:=\sum_{p\in\{p_0,p_1\}}\sup_{J\subseteq I,\vert J\vert\le 1}(\sum_{N=2^k,k\in\mathbb{N}}N^{5-p/2}\Vert P_Nu(t)\Vert_{L^{p}_{x,t}(\mathbb{S}^3\times J)}^{p})^{1/p},\\
&p_0=4+1/10,\qquad p_1=100.
\end{split}
\end{equation}
This definition, in particular the choice of the exponents $p_0,p_1$, is motivated by the Strichartz estimates from Theorem \ref{Stric} below.
This norm is divisible and, thanks to sufficiently strong multilinear Strichartz estimates, still controls the global evolution, as will be manifest from the local theory in Section \ref{localwp}. 
Moreover, as a consequence of Corollary \ref{NewS31} below,
\begin{equation*}
\Vert u\Vert_{Z(I)}\lesssim \Vert u\Vert_{X^1(I)},
\end{equation*}
thus $Z$ is indeed a weaker norm. 

{\bf Definition of solutions}.
Given an interval $I\subseteq\mathbb{R}$, we call $u\in C(I:H^1(\mathbb{S}^3))$ a strong solution of \eqref{NLSBis} if $u\in X^1(I)$ and $u$ satisfies that for all $t,s\in I$,
\begin{equation*}
u(t)=e^{i(t-s)L}u(s)+i\int_s^te^{i(t-t^\prime)L}\left(u(t^\prime)\vert u(t^\prime)\vert^4\right)dt^\prime.
\end{equation*}

{\bf Dispersive estimates}. We recall the following result from \cite[Lemma 3.5.]{He}.

\begin{theorem}\label{Stric}
If $p>4$ then
\begin{equation*}
\Vert P_Ne^{itL}f\Vert_{L^p_{x,t}(\mathbb{S}^3\times [-1,1])}\lesssim_p N^{\frac{3}{2}-\frac{5}{p}}\Vert P_Nf\Vert_{L^2(\mathbb{S}^3)}.
\end{equation*}
\end{theorem}

As a consequence of the properties of the $U^p_L$ spaces, we have:

\begin{corollary}\label{NewS31}
If $p>4$ then for any dyadic integer $N$ and any time interval $I$, $\vert I\vert\leq 1$,
\begin{equation}\label{UpEst}
\Vert P_Nu\Vert_{L^p_{x,t}(\mathbb{S}^3\times I)}\lesssim N^{\frac{3}{2}-\frac{5}{p}}\Vert u\Vert_{U^p_L(I,L^2)}.
\end{equation}
\end{corollary}

We will also use the following results from Herr \cite{He}.

\begin{proposition}[\cite{He}, Lemma  2.5]\label{Alex3}
If $f\in L^1_t(I,H^1(\mathbb{S}^3))$ then
\begin{equation}\label{EstimNnorm}
\Vert f\Vert_{N(I)}\lesssim \sup_{\{\Vert v\Vert_{Y^{-1}(I)}\le1\}}\int_{\mathbb{S}^3\times I}f(x,t)\overline{v(x,t)}dxdt.
\end{equation}
In particular, there holds for any smooth function $g$ that
\begin{equation}\label{EstimX1Norm}
\Vert g\Vert_{X^1([0,1])}
\lesssim \Vert g(0)\Vert_{H^1}+(\sum_N \Vert P_N\left(i\partial_t+L\right)g\Vert_{L^1_t([0,1], H^1)}^2)^\frac{1}{2}.
\end{equation}
\end{proposition}


\section{Local well-posedness and stability theory}\label{localwp}

In this section we present large-data local well-posedness and stability results that allow us to connect nearby intervals of nonlinear evolution. This is essentially a modification of the results in \cite{He}.
We need the following notation
\begin{equation}\label{Zprime}
\Vert u\Vert_{Z^\prime(I)}=\Vert u\Vert_{Z(I)}^\frac{1}{2}\Vert u\Vert_{X^1(I)}^\frac{1}{2}.
\end{equation}

We start with the following nonlinear estimate:
\begin{lemma}\label{NewS32}
There exists $\delta>0$ such that if $u_1,u_2,u_3$ satisfy $P_{N_i}u_i=u_i$ with $N_1\ge N_2\ge N_3\ge 1$ and $\vert I\vert\le 1$, then
\begin{equation}\label{NLEstTor}
\Vert u_1u_2u_3\Vert_{L^2_{x,t}(\mathbb{S}^3\times I)}\lesssim \left(\frac{N_3}{N_1}+\frac{1}{N_2}\right)^\delta \Vert u_1\Vert_{Y^0(I)}\Vert u_2\Vert_{Z^\prime(I)}\Vert u_3\Vert_{Z^\prime(I)}
\end{equation}
and, with $p_0=4+1/10$ as in \eqref{NewS319},  
\begin{equation}\label{Al1}
\Vert u_1u_2u_3\Vert_{L^2_{x,t}(\mathbb{S}^3\times I)}\lesssim N_1^{1/2-5/p_0}N_2^{1/2-5/p_0}N_3^{10/p_0-2}\Vert u_1\Vert_{Z(I)}\Vert u_2\Vert_{Z(I)}\Vert u_3\Vert_{Z(I)}.
\end{equation}
\end{lemma}

\begin{proof}[Proof of Lemma \ref{NewS32}]
Inequality \eqref{NLEstTor} follows from interpolation between the two estimates
\begin{equation*}
\begin{split}
\Vert u_1u_2u_3\Vert_{L^2_{x,t}(\mathbb{S}^3\times I)}&\lesssim \left(\frac{N_3}{N_1}+\frac{1}{N_2}\right)^\delta N_2N_3\Vert u_1\Vert_{V^2_L(I)}\Vert u_2\Vert_{V^2_L(I)}\Vert u_3\Vert_{V^2_L(I)},\\
\Vert u_1u_2u_3\Vert_{L^2_{x,t}(\mathbb{S}^3\times I)}&\lesssim \Vert u_1\Vert_{V^2_L(I)}\left(\Vert u_2\Vert_{Z(I)}\Vert u_3\Vert_{Z(I)}\right)^\frac{3}{5}\left(\Vert u_2\Vert_{X^1(I)}\Vert u_3\Vert_{X^1(I)}\right)^\frac{2}{5}.
\end{split}
\end{equation*}
The first is taken directly from \cite[Corollary 3.7]{He}, while the second follows from the following modifications of its proof.
We start with the estimate
\begin{equation}\label{VarHe1}
\begin{split}
\Vert u_1u_2u_3\Vert_{L^2_{x,t}}&\lesssim \left[\max(N_2^2/N_1,1)\right]^{1/2-2/p_1}N_2^{1+\varepsilon-2/p_2}N_3^{\frac{3}{2}-\varepsilon-\frac{2}{p_3}}\Vert u_1\Vert_{U^2_L}\Vert u_2\Vert_{U^2_L}\Vert u_3\Vert_{U^2_L}\\
\end{split}
\end{equation}
valid for $\varepsilon>0$ and $4<p_1,p_2,p_3<+\infty$ satisfying $1/p_1+1/p_2+1/p_3=1/2$ which we borrow from the proof of
\cite[Proposition 3.6]{He}. Independently, using Theorem \ref{Stric} and H\"older's inequality, we obtain
\begin{equation}\label{VarHe2}
\begin{split}
\Vert u_1u_2u_3\Vert_{L^2_{x,t}}&\lesssim \Vert u_1\Vert_{L^{q_1}_{x,t}}\Vert u_2\Vert_{L^{q_2}_{x,t}}\Vert u_3\Vert_{L^{q_3}_{x,t}}\\
&\lesssim N_1^{\frac{3}{2}-\frac{5}{q_1}}N_2^{\frac{1}{2}-\frac{5}{q_2}}N_3^{\frac{1}{2}-\frac{5}{q_3}}\Vert u_1\Vert_{U^{q_1}_L}\left(N_2^{\frac{5}{q_2}-\frac{1}{2}}\Vert u_2\Vert_{L^{q_2}_{x,t}}\right)\left(N_3^{\frac{5}{q_3}-\frac{1}{2}}\Vert u_3\Vert_{L^{q_3}_{x,t}}\right)
\end{split}
\end{equation}
where $4<q_1,q_2,q_3<+\infty$ satisfy $1/q_1+1/q_2+1/q_3=1/2$.

\medskip

In the case $N_1\le N_2^2$, we may choose
\begin{equation*}
\begin{split}
p_1=40,\quad q_1=q_2=p_2=25/6,\quad p_3=200/47,\quad q_3=50,\quad\varepsilon=1/100
\end{split}
\end{equation*} 
and apply \cite[Lemma 2.4]{He}.

\medskip

In the case $N_2^2\le N_1$, we use \eqref{VarHe1} with the same exponents, while \eqref{VarHe2} is replaced by
\begin{equation*}
\begin{split}
\Vert u_1u_2u_3\Vert_{L^2_{x,t}}&\lesssim \Vert u_1\Vert_{L^\infty_tL^2_x}\Vert u_2\Vert_{L^{p_2}_tL^\infty_x}\Vert u_3\Vert_{L^{p_2}_tL^\infty_x}\\
&\lesssim (N_2N_3)^{\frac{3}{p_2}}\Vert u_1\Vert_{L^\infty_tL^2_x}\Vert u_2\Vert_{L^{p_2}_{x,t}}\Vert u_3\Vert_{L^{p_2}_{x,t}}\\
&\lesssim (N_2N_3)^{\frac{1}{2}-\frac{2}{p_2}}\Vert u_1\Vert_{U^4_L}\Vert u_2\Vert_{Z}\Vert u_3\Vert_{Z}\\
\end{split}
\end{equation*}
and we apply again \cite[Lemma 2.4]{He}.

Finally, \eqref{Al1} follows from \eqref{VarHe2} with $q_1=q_2=p_0$ and $q_3=20p_0$. 
\end{proof}

From here on, we have an estimate formally identical to the nonlinear estimate in \cite[Lemma 3.1.]{IoPa2} and the following lemma and propositions are proved using straightforward adaptation from \cite[Section 3]{IoPa2} (see also \cite{IoPa}).

\begin{lemma}\label{NLEst2}
For $u_k\in X^1(I)$, $k=1\dots 5$, $\vert I\vert\leq 1$, the estimate
\begin{equation*}
\Vert \Pi_{i=1}^5\tilde{u}_k\Vert_{N(I)}\lesssim \sum_{\sigma\in\mathfrak{S}_5}\Vert u_{\sigma(1)}\Vert_{X^1(I)}\Pi_{j\ge 2}\Vert u_{\sigma(j)}\Vert_{Z^\prime(I)}
\end{equation*}
holds true, where $\tilde{u}_k\in\{u_k,\overline{u_k}\}$. In fact, we have that
\begin{equation}\label{NLEst3}
\Vert \sum_{B\ge 1}P_{B}\tilde{u}_1\Pi_{j=2}^5P_{\le DB}\tilde{u}_j\Vert_{N(I)}\lesssim_D \Vert u_1\Vert_{X^1(I)}\Pi_{j=2}^5\Vert u_j\Vert_{Z^\prime(I)},
\end{equation}
\end{lemma}

We have a local existence result:

\begin{proposition}[Local well-posedness]\label{LWP}
(i) Given $E>0$, there exists $\delta_0=\delta_0(E)>0$ such that if $\Vert \phi\Vert_{H^1(\mathbb{S}^3)}\le E$ and
\begin{equation*}
\Vert e^{itL}\phi\Vert_{Z(I)}\le\delta_0
\end{equation*}
on some interval $I\ni 0$, $\vert I\vert\leq 1$, then there exists a unique solution $u\in X^1(I)$ of \eqref{NLSBis} satisfying $u(0)=\phi$. Besides
\begin{equation*}
\Vert u-e^{itL}\phi\Vert_{X^1(I)}\lesssim_E \Vert e^{itL}\phi\Vert_{Z(I)}^{3/2}.
\end{equation*}
The quantities $E(u)$ and $M(u)$ defined in \eqref{conserve} are conserved on $I$. 

(ii) If $u\in X^1(I)$ is a solution of \eqref{NLSBis} on some open interval $I$ and 
\begin{equation*}
\Vert u\Vert_{Z(I)}<+\infty
\end{equation*}
then $u$ can be extended as a nonlinear solution to a neighborhood of $\overline{I}$ and
\begin{equation*}
\Vert u\Vert_{X^1(I)}\le C(E(u),\Vert u\Vert_{Z(I)})
\end{equation*}
for some constant $C$ depending on $E(u)$ and $\Vert u\Vert_{Z(I)}$.
\end{proposition}

The main result in this section is the following:

\begin{proposition}[Stability]\label{Stabprop}
Assume $I$ is an open bounded interval, $\rho\in[-1,1]$, and $\widetilde{u}\in X^1(I)$ satisfies the approximate  Schr\"{o}dinger equation
\begin{equation}\label{ANLS}
(i\partial_t+L)\widetilde{u}+\rho\widetilde{u}|\widetilde{u}|^4=e\quad\text{ on }\mathbb{S}^3\times I.
\end{equation}
Assume in addition that
\begin{equation}\label{ume}
\|\widetilde{u}\|_{Z(I)}+\|\widetilde{u}\|_{L^\infty_t(I,H^1(\mathbb{S}^3))}\leq M,
\end{equation}
for some $M\in[1,\infty)$. Assume $t_0 \in I$ and $u_0\in H^1(\mathbb{S}^3)$ is such that the smallness condition
\begin{equation}\label{safetycheck}
\|u_0 - \widetilde{u}(t_0)\|_{H^1(\mathbb{S}^3)}+\| e\|_{N(I)}\leq \epsilon
\end{equation}
holds for some $0 < \epsilon < \epsilon_1$, where $\epsilon_1\leq 1$ is a small constant $\epsilon_1 = \epsilon_1(M) > 0$.

Then there exists a strong solution $u\in X^1(I)$ of the Schr\"{o}dinger equation
\begin{equation}\label{ANLS2}
(i\partial_t+L)u+\rho u|u|^4=0
\end{equation}
 such that $u(t_0)=u_0$ and
\begin{equation}\label{output}
\begin{split}
\| u \|_{X^1(I)}+\|\widetilde{u}\|_{X^1(I)}&\leq C(M),\\
\| u - \widetilde u \|_{X^1(I)}&\leq C(M)\epsilon.
 \end{split}
\end{equation}
\end{proposition}


\section{Profiles}\label{SecProf}

\subsection{Analysis of Euclidean profiles}

In this section we prove precise estimates showing how to compare Euclidean and spherical solutions of both linear and nonlinear Schr\"{o}dinger equations. Of course, such a comparison is only meaningful in the case of rescaled data that concentrate at a point. We follow closely the arguments in \cite{IoPa,IoPaSt}, the main novelty being in Lemma \ref{Extinction}.

Recall $\eta$ defined\footnote{The role of $\eta$ is to avoid ``tail'' effects coming from the fact that $\phi$ might not vanish outside of $B(0,R)$ for any $R$.} in \eqref{DefPN}. Given $\phi\in \dot{H}^1(\mathbb{R}^3)$ and a real number $N\geq 1$ we define
\begin{equation}\label{rescaled}
\begin{split}
T_N\phi=f_{N}\in H^1(\mathbb{S}^3),\qquad &f_{N}(y)=N^\frac{1}{2}\eta(N^{1/2}d_g(O,y))\phi(N\exp_O^{-1}(y)).
\end{split}
\end{equation}
and observe that
\begin{equation*}
T_N:\dot{H}^1(\mathbb{R}^3)\to H^1(\mathbb{S}^3)\text{ is a linear operator with }\|T_N\phi\|_{H^1(\mathbb{S}^3)}\lesssim \|\phi\|_{\dot{H}^1(\mathbb{R}^3)}
\end{equation*}
and that
\begin{equation*}
\Vert T_N\phi\Vert_{L^1}\lesssim N^{-\frac{5}{2}}\Vert \phi\Vert_{L^1},\quad\Vert T_N\phi\Vert_{L^2}\lesssim N^{-1}\Vert \phi\Vert_{L^2}.
\end{equation*}
We define also
\begin{equation*}
E_{\mathbb{R}^3}(\phi)=\frac{1}{2}\int_{\mathbb{R}^3}\left[|\nabla_{\mathbb{R}^3}\phi|^2+\frac{1}{3}|\phi|^6\right]\,dx.
\end{equation*}

We will use the main theorem of \cite{CKSTTcrit} (see also \cite{KiVi} and \cite{B,G,KeMe} for previous results), in the following form.

\begin{theorem}\label{MainThmEucl}
Assume $\psi\in\dot{H}^1(\mathbb{R}^3)$. Then there is a unique global solution $v\in C(\mathbb{R}:\dot{H}^1(\mathbb{R}^3))$ of the initial-value problem
\begin{equation}\label{clo3}
(i\partial_t-\Delta_{\mathbb{R}^3})v+v|v|^4=0,\qquad v(0)=\psi,
\end{equation}
and
\begin{equation}\label{clo4}
\Vert v\Vert_{L^4_tL^\infty_x(\mathbb{R}^3\times\mathbb{R}))}+\|\nabla_{\mathbb{R}^3} v\|_{(L^\infty_tL^2_x\cap L^2_tL^6_x)(\mathbb{R}^3\times\mathbb{R})}\leq \widetilde{C}(E_{\mathbb{R}^3}(\psi)).
\end{equation}
Moreover this solution scatters in the sense that there exists $\psi^{\pm\infty}\in\dot{H}^1(\mathbb{R}^3)$ such that
\begin{equation}\label{EScat}
\Vert v(t)-e^{-it\Delta}\psi^{\pm\infty}\Vert_{\dot{H}^1(\mathbb{R}^3)}\to 0
\end{equation}
as $t\to\pm\infty$. Besides, if $\psi\in H^5(\mathbb{R}^3)$ then $v\in C(\mathbb{R}:H^5(\mathbb{R}^3))$ and
\begin{equation*}
\sup_{t\in\mathbb{R}}\|v(t)\|_{H^5(\mathbb{R}^3)}\lesssim_{\|\psi\|_{H^5(\mathbb{R}^3)}}1.
\end{equation*}
\end{theorem}
Again, we emphasize that this extends readily to the case when $-\Delta_{\mathbb{R}^3}$ is replaced by $1-\Delta_{\mathbb{R}^3}$.

Our first result in this section is the following lemma:

\begin{lemma}\label{step1}
Assume $\phi\in\dot{H}^1(\mathbb{R}^3)$, $T_0\in(0,\infty)$, and $\rho\in\{0,1\}$ are given, and define $f_{N}$ as in \eqref{rescaled}. Then the following conclusions hold:

(i) There is $N_0=N_0(\phi,T_0)$ sufficiently large such that for any $N\geq N_0$ there is a unique solution $U_{N}\in C((-T_0N^{-2},T_0N^{-2}):H^1(\mathbb{S}^3))$ of the initial-value problem
\begin{equation}\label{clo5}
(i\partial_t-\Delta+1)U_N=\rho U_N|U_N|^4,\qquad U_N(0)=f_N.
\end{equation}

(ii) Assume $\varepsilon_1\in(0,1]$ is sufficiently small (depending only on $E_{\mathbb{R}^3}(\phi)$), $\phi'\in H^5(\mathbb{R}^3)$, and $\|\phi-\phi'\|_{\dot{H}^1(\mathbb{R}^3)}\leq\varepsilon_1$. Let $v'\in C(\mathbb{R}:H^5(\mathbb{R}^3))$ denote the solution of the initial-value problem
\begin{equation*}
(i\partial_t-\Delta_{\mathbb{R}^3}+1)v'=\rho v'|v'|^4,\qquad v'(0)=\phi'.
\end{equation*}
For $R,N\geq 1$ we define
\begin{equation}\label{clo9}
\begin{split}
&v'_R(x,t)=\eta(\vert x\vert/R)v'(x,t),\qquad\,\,\qquad (x,t)\in\mathbb{R}^3\times(-T_0,T_0),\\
&v'_{R,N}(x,t)=N^\frac{1}{2}v'_R(Nx,N^2t),\qquad\quad\,(x,t)\in\mathbb{R}^3\times(-T_0N^{-2},T_0N^{-2}),\\
&V_{R,N}(y,t)=v'_{R,N}(\exp_O^{-1}(y),t)\qquad\quad\,\, (y,t)\in\mathbb{S}^3\times(-T_0N^{-2},T_0N^{-2}).
\end{split}
\end{equation}
Then there is $R_0\geq 1$ (depending on $T_0$ and $\phi'$ and $\varepsilon_1$) such that, for any $R\geq R_0$,
\begin{equation}\label{clo18}
\limsup_{N\to\infty}\|U_N-V_{R,N}\|_{X^1(-T_0N^{-2},T_0N^{-2})}\lesssim_{E_{\mathbb{R}^3}(\phi)}\varepsilon_1.
\end{equation}
In particular, for any $N\geq N_0$,
\begin{equation}\label{clo6}
\|U_N\|_{X^1(-T_0N^{-2},T_0N^{-2})}\lesssim_{E_{\mathbb{R}^3}(\phi)}1.
\end{equation}

\end{lemma}

\begin{remark}
As is shown in \cite[Appendix A]{BuGeTz3} (see also \cite{ChCoTa}), for times $0\le t\ll N^{-2}$, the effect of the dispersion is weak and a good approximation for \eqref{NLSBis} is the simple $ODE$
\begin{equation*}
i\partial_tu=u\vert u\vert^4-u.
\end{equation*}
This lemma shows how to take into account the effect of the dispersion on the interval $[N^{-2},TN^{-2}]$ for $T$ large, so as to complement the conclusion of Lemma \ref{Extinction} below. 
\end{remark}

\begin{proof}[Proof of Lemma \ref{step1}] In fact, we show that $V_{R,N}$ in $(ii)$ gives such a good ansatz that we can apply the stability Proposition \ref{Stabprop} and obtain \eqref{clo18}, which in particular implies $(i)$. All of the constants in this proof are allowed to depend on $E_{\mathbb{R}^3}(\phi)$. Using Theorem \ref{MainThmEucl}
\begin{equation}\label{clo7}
\begin{split}
&\Vert v^\prime\Vert_{L^4_tL^\infty_x(\mathbb{R}\times\mathbb{R}^3)}+\|\nabla_{\mathbb{R}^3} v'\|_{(L^\infty_tL^2_x\cap L^2_tL^6_x)(\mathbb{R}^3\times\mathbb{R})}\lesssim 1,\\
&\sup_{t\in\mathbb{R}}\|v'(t)\|_{H^5(\mathbb{R}^3)}\lesssim_{\|\phi'\|_{H^5(\mathbb{R}^3)}}1.
\end{split}
\end{equation}
Let
\begin{equation*}
\begin{split}
e_R(x,t):&=[(i\partial_t-\Delta_{\mathbb{R}^3}+1)v'_R-\rho v'_R|v'_R|^4](x,t)=\rho(\eta(\vert x\vert/R)-\eta(\vert x\vert/R)^5)v'(x,t)|v'(x,t)|^4\\
&-R^{-2}v'(x,t)\eta^{\prime\prime}(\vert x\vert/R)-2R^{-1}\vert x\vert^{-1}v'(x,t)\eta^{\prime}(\vert x\vert/R)-2R^{-1}\sum_{j=1}^4\partial_rv'(x,t)\eta^\prime(\vert x\vert/R).
\end{split}
\end{equation*}
Since $|v'(x,t)|\lesssim_{\|\phi'\|_{H^5(\mathbb{R}^3)}}1$, see \eqref{clo7}, it follows that
\begin{equation*}
\begin{split}
|e_R(x,t)|&+\sum_{k=1}^3|\partial_ke_R(x,t)|\\
&\lesssim_{\|\phi'\|_{H^5(\mathbb{R}^3)}}\mathbf{1}_{[R,2R]}(|x|)\cdot\big[|v'(x,t)|+\sum_{k=1}^3|\partial_kv'(x,t)|+\sum_{k,j=1}^3|\partial_k\partial_jv'(x,t)|\big].
\end{split}
\end{equation*}
Therefore
\begin{equation}\label{clo10}
\lim_{R\to\infty}\|\,|e_R|+|\nabla_{\mathbb{R}^3} e_R|\,\|_{L^\infty_tL^2_x(\mathbb{R}^3\times(-T_0,T_0))}=0.
\end{equation}
Letting
\begin{equation*}
e_{R,N}(x,t):=[(i\partial_t-\Delta_{\mathbb{R}^3}+1)v'_{R,N}-\rho v'_{R,N}|v'_{R,N}|^4](x,t)=N^\frac{5}{2}e_R(Nx,N^2t),
\end{equation*}
it follows from \eqref{clo10} that there is $R_0\geq 1$ such that, for any $R\geq R_0$ and $N\geq 1$,
\begin{equation}\label{clo11}
\|\,|e_{R,N}|+|\nabla_{\mathbb{R}^3} e_{R,N}|\,\|_{L^1_tL^2_x(\mathbb{R}^3\times(-T_0N^{-2},T_0N^{-2}))}\leq\varepsilon_1.
\end{equation}

With $V_{R,N}(y,t)=v'_{R,N}(\exp_O^{-1}(y),t)$ as in \eqref{clo9} and $N\geq 10R$, let
\begin{equation}\label{clo13}
\begin{split}
E_{R,N}(y,t)&:=[(i\partial_t+L)V_{R,N}-\rho V_{R,N}|V_{R,N}|^4](y,t)\\
&=e_{R,N}(\exp_O^{-1}(y),t)+2(1/\phi-1/\sin \phi)(\partial_\phi v^\prime_{R,N})(\exp_O^{-1}(y),t)\\
&+(1/\phi^2-1/\sin^2\phi)(\Delta_{\mathbb{S}^2}v^\prime_{R,N})(\exp_O^{-1}(y),t)
\end{split}
\end{equation}
where we have used the formula in \eqref{Delta}.
We remark that
\begin{equation*}
\begin{split}
\Vert \phi \partial_\phi v^\prime_{R,N}(\exp_O^{-1}(y),t)\Vert_{L^1_tL^2_x}+\Vert \phi\nabla (\partial_\phi v^\prime_{R,N})(\exp_O^{-1}(y),t)\Vert_{L^1_tL^2_x}\lesssim_{R,T} N^{-2}\\
\Vert \Delta_{\mathbb{S}^2}v^\prime_{R,N}(\exp_O^{-1}(y),t)\Vert_{L^1_tL^2_x}+\Vert\nabla (\Delta_{\mathbb{S}^2}v^\prime_{R,N})(\exp_O^{-1}(y),t)\Vert_{L^1_tL^2_x}\lesssim_{R,T} N^{-2}.
\end{split}
\end{equation*}

Using \eqref{clo11}, it follows that for any $R_0$ sufficiently large there is $N_0$ such that for any $N\geq N_0$
\begin{equation}\label{clo15}
\|\,|\nabla^1 E_{R_0,N}|\,\|_{L^1_tL^2_x(\mathbb{S}^3\times(-T_0N^{-2},T_0N^{-2}))}\leq 2\varepsilon_1.
\end{equation}

To verify the hypothesis \eqref{ume} of Proposition \ref{Stabprop}, we estimate for $N$ large enough, using \eqref{clo7}
\begin{equation}\label{clo16}
\sup_{t\in(-T_0N^{-2},T_0N^{-2})}\|V_{R_0,N}(t)\|_{H^1(\mathbb{S}^3)}\leq \sup_{t\in(-T_0N^{-2},T_0N^{-2})}\|v'_{R_0,N}(t)\|_{H^1(\mathbb{R}^3)}\lesssim 1.
\end{equation}
and using \eqref{EstimX1Norm}, \eqref{clo15} and
\begin{equation*}
\Vert V_{R,N}\vert V_{R,N}\vert^4\Vert_{L^1_tH^1}\lesssim \Vert v^\prime\Vert_{L^4L^\infty_x}^4\Vert v^\prime\Vert_{L^\infty_tH^1_x}\lesssim 1
\end{equation*}
we obtain that
\begin{equation*}
\Vert V_{R,N}\Vert_{X^1}\lesssim 1.
\end{equation*}

Finally, to verify the inequality on the first term in \eqref{safetycheck} we estimate, for $R_0,N$ large enough,
\begin{equation}\label{clo17}
\begin{split}
\|f_N-V_{R_0,N}(0)\|_{H^1(\mathbb{S}^3)}&\lesssim \|\phi_N-v'_{R_0,N}(0)\|_{\dot{H}^1(\mathbb{R}^3)}\lesssim\|\eta(N^\frac{1}{2}\cdot)\phi-v'_{R_0}(0)\|_{\dot{H}^1(\mathbb{R}^3)}\\
&\lesssim \|(1-\eta(N^\frac{1}{2}\cdot))\phi\|_{\dot{H}^1(\mathbb{R}^3)}+\|\phi-\phi'\|_{\dot{H}^1(\mathbb{R}^3)}\\
&\quad+\|\phi'-v'_{R_0}(0)\|_{\dot{H}^1(\mathbb{R}^3)}\\
&\lesssim \varepsilon_1.
\end{split}
\end{equation}
The conclusion of the lemma follows from Proposition \ref{Stabprop}, provided that $\varepsilon_1$ is fixed sufficiently small depending on $E_{\mathbb{R}^3}(\phi)$.
\end{proof}


To understand linear and nonlinear evolutions beyond the Euclidean window we need an additional extinction lemma:

\begin{lemma}\label{Extinction}
Let $\phi\in \dot{H}^1(\mathbb{R}^3)$ and define $f_N$ as in \eqref{rescaled}. For any $\varepsilon>0$, there exists $T=T(\phi,\varepsilon)$ and $N_0(\phi,\varepsilon)$ such that for all $N\ge N_0$, there holds that
\begin{equation}\label{Lem4.4Stat}
\Vert e^{itL}f_{N}\Vert_{Z(TN^{-2},T^{-1})}\lesssim\varepsilon.
\end{equation}
\end{lemma}

\begin{remark}
Note that the analysis in \cite{BuGeTz} already gives the result on an interval of time of the form $[TN^{-2},N^{-1}]$. However for our application, it is important to obtain an upper bound independent of $N$.
\end{remark}

\begin{proof}

Using Strichartz estimates and interpolation, we see that it suffices to obtain this for $p=\infty$ in the definition of $Z$, i.e.
\begin{equation*}
\sup_{M}M^{-\frac{1}{2}}\Vert P_Me^{itL}f_N\Vert_{L^\infty_{x,t}(\mathbb{S}^3\times[TN^{-2},T^{-1}])}\lesssim\varepsilon.
\end{equation*}
Fix $\varphi\in C^\infty_c(\mathbb{R}^3)$ such that
\begin{equation*}
\Vert \phi-\varphi\Vert_{\dot{H}^1(\mathbb{R}^3)}\le \varepsilon^2.
\end{equation*}
From the boundedness of $T_N$ in \eqref{rescaled}, we deduce that it suffices to prove that
\begin{equation*}
\sup_MM^{-1/2}\Vert P_Me^{itL}\varphi_N\Vert_{L^\infty_{x,t}}\le \varepsilon,\quad\varphi_N=T_N\varphi.
\end{equation*}
Let $Q=R^2+\varepsilon^{-2}$, where $R$ is the diameter of the support of $\varphi$. Using Bernstein estimate, we observe that
\begin{equation}\label{OffScale}
\begin{split}
M^{-\frac{1}{2}}\Vert P_Me^{itL}\varphi_N\Vert_{L^\infty_{x,t}}&\lesssim M\Vert P_Me^{itL}\varphi_N\Vert_{L^\infty_tL^2_x}\lesssim \min\left(\frac{M}{N},\left(\frac{N}{M}\right)^{10}\right).
\end{split}
\end{equation}
Thus, if $(M/N)\notin(Q^{-1},Q)$, \eqref{Lem4.4Stat} holds. From now on, we assume that
\begin{equation*}
Q^{-1}\le M/N\le Q.
\end{equation*}

We define
\begin{equation*}
c_p(x)=\left[\pi_p\varphi_N\right](x).
\end{equation*}
This decouples the oscillations in time and the variations in space as follows:
\begin{equation}\label{Explicit1111}
P_Me^{itL}\varphi_N(x)=\sum_{p\le 2M}\eta_M(p)e^{it p^2}c_p(x).
\end{equation}

We consider two cases.

\medskip

{\bf Case 1:} when $d_g(O,x)\ge Q^6/N$. In this case, we can use the explicit formula \eqref{Zonal} to get that the function is far from saturating Sobolev inequality
\begin{equation}\label{SmallSobEmb}
\sum_{M\le p\le 2M}\vert\pi_p(\varphi_N)(x)\vert\lesssim\varepsilon N^\frac{1}{2}.
\end{equation}

\medskip

From the formula \eqref{DefSpecProj} and the fact that in our case, for any $Y$ in the support of $\varphi_N$, $\angle(Y,x)\ge Q^5/N$  we obtain that
\begin{equation*}
\begin{split}
\vert  \pi_p(\varphi_N)(x)\vert&\lesssim \Vert \varphi_N\Vert_{L^1}p(N/Q^{5})\lesssim \varepsilon^{2}Q^{-4} N^{-\frac{3}{2}}p.
\end{split}
\end{equation*}
Summing crudely over all $p\le 2M$, we obtain that
\begin{equation*}
\vert \sum_{p\le 2M}\eta_M(p)e^{-it p^2}c_p(x)\vert\lesssim N^{-\frac{3}{2}}Q^{-4}\sum_{p\le 2M}\varepsilon^{2}p\le\varepsilon N^\frac{1}{2}.
\end{equation*}
which gives \eqref{Lem4.4Stat} in this case.

\medskip

{\bf Case 2:} when $d_g(O, x)\le 2Q^6/N$. In this case, we claim that, uniformly in $p$, $d_g(O, x)$, there holds that
\begin{equation}\label{EstimFourierCoeff}
\begin{split}
\vert c_p(x)\vert&\lesssim_\varphi Q^{10}N^{-\frac{1}{2}},\\
\vert c_p(x)-c_{p-1}(x)\vert&\lesssim_\varphi Q^{10}N^{-\frac{3}{2}},\\
\vert c_p(x)-2c_{p-1}(x)+c_{p-2}(x)\vert&\lesssim_\varphi Q^{10}N^{-\frac{5}{2}}.
\end{split}
\end{equation}
This follows from the explicit formulas
\begin{equation*}
\begin{split}
c_p( Q)&=\int_{\mathbb{S}^3}{\bf Z}_p( R_QP)\varphi_N( P)d\nu_g( P),\\
c_p(Q)-c_{p-1}(Q)&=\int_{\mathbb{S}^3}{\bf Z}_p^d( R_QP)\varphi_N( P)d\nu_g( P),\\
c_p(Q)-2c_{p-1}(Q)+c_{p-2}(Q)&=\int_{\mathbb{S}^3}{\bf Z}_p^{dd}( R_QP)\varphi_N( P)d\nu_g( P),
\end{split}
\end{equation*}
where
\begin{equation*}
\begin{split}
\vert Z_p(\theta)\vert=&p\frac{\vert \sin(p\theta)\vert}{\sin\theta}\lesssim p^2\\
\vert Z_p^d(\theta)\vert =&p\left\vert \sin(p\theta)\frac{1-\cos\theta}{\sin\theta}+\cos(p\theta)+\frac{\sin((p-1)\theta)}{p\sin\theta}\right\vert\lesssim p(1+p\theta)\\
\left\vert Z_p^{dd}(\theta)\right\vert =&(p-1)\Big\vert \frac{\sin(p\theta)}{\sin\theta}\left[1-2\cos\theta+\cos(2\theta)\right]\\
&+\cos(p\theta)\left[2-\frac{\sin(2\theta)}{\sin\theta}\right]+\frac{2}{p-1}\frac{\cos\theta-\cos(2\theta)}{\sin\theta}+\frac{\sin p\theta}{p\sin\theta}\left[1-\cos(2\theta)\right]\Big\vert\\
\lesssim& p^2\theta^2+\theta.
\end{split}
\end{equation*}

We may now use \eqref{Explicit1111}, \eqref{EstimFourierCoeff} together with Lemma \ref{BouLem} (with $K=Q^{10}N^{-\frac{1}{2}}$) to find an acceptable $T$ as in \eqref{Lem4.4Stat}. More precisely, we fix $T_0\ge\varepsilon^{-3}$, which forces either $(a,q)=(0,1)$ or $q\ge \varepsilon^{-2}$ and then find choose $T\ge T_0$ in such a way as to satisfy \eqref{Lem4.4Stat}.
\end{proof}

In the process, we have seen from \eqref{OffScale}, \eqref{SmallSobEmb} and the end of the proof above that if $\varphi\in C^\infty_c(\mathbb{R}^3)$, then, for any $\varepsilon$, there exists $T_0>0$ and $N_0$ such that, whenever $T\ge T_0$ and $N\ge N_0$, there holds that 
\begin{equation}\label{ExplicitDisp}
\sum_{M\ge 1}M^{-1/2}\Vert e^{itL}P_M(T_N\varphi)\Vert_{L^\infty(\mathbb{S}^3\times (TN^{-2},T^{-1}))}\lesssim\varepsilon.
\end{equation}

We conclude this section with a proposition describing nonlinear solutions of the initial-value problem \eqref{NLSBis} corresponding to data concentrating at a point. In view of the profile analysis in the next section, we need to consider slightly more general data. Given $f\in L^2(\mathbb{S}^3)$, $t_0\in\mathbb{R}$ and $x_0\in\mathbb{S}^3$ we define
\begin{equation*}
\begin{split}
&(\Pi_{t_0,x_0})f(x)=(e^{-it_0L}\tau_{x_0}f)(x),
\end{split}
\end{equation*}
where $\tau_{x_0}f(x)=f(R_{x_0}x)$.

Let $\widetilde{\mathcal{F}}_e$ denote the set of renormalized Euclidean frames
\footnote{We will later consider a slightly more general class of frames, called {\it{Euclidean frames}}, see Definition \ref{DefPro}. 
For our later application, it suffices to prove Proposition \ref{GEForEP} under the stronger assumption that $\mathcal{O}$ is a renormalized Euclidean frame.} 
\begin{equation}\label{renframe}
\begin{split}
\widetilde{\mathcal{F}}_e:=\{(N_k,t_k,x_k)_{k\geq 1}:&\,N_k\in[1,+\infty),\,t_k\to 0,\,x_k\in\mathbb{S}^3,\,N_k\to+\infty,\\
&\text{ and either }t_k=0 \text{ for any }k\geq 1\text{ or }\lim_{k\to\infty}N_k^2|t_k|=+\infty\}.
\end{split}
\end{equation}

\begin{proposition}\label{GEForEP}
Assume that $\mathcal{O}=(N_k,t_k,x_k)_k\in\widetilde{\mathcal{F}}_e$, $\phi\in\dot{H}^1(\mathbb{S}^3)$, and let $U_k(0)=\Pi_{t_k,x_k}(T_{N_k}\phi)$.

(i) There exists $\tau=\tau(\phi)$ such that for $k$ large enough (depending only on $\phi$ and $\mathcal{O}$) 
there is a nonlinear solution $U_k\in X^1(-\tau,\tau)$ of the equation \eqref{NLSBis} with initial data $U_k(0)$, and
\begin{equation}\label{ControlOnZNormForEP}
\Vert U_k\Vert_{X^1(-\tau,\tau)}\lesssim_{E_{\mathbb{R}^3}(\phi)}1.
\end{equation}

(ii) There exists a Euclidean solution $u\in C(\mathbb{R}:\dot{H}^1(\mathbb{R}^3))$ of
\begin{equation}\label{EEq}
\left(i\partial_t-\Delta_{\mathbb{R}^3}+1\right)u+u\vert u\vert^4=0
\end{equation}
with scattering data $\phi^{\pm\infty}$ defined as in \eqref{EScat} such that the following holds, up to a subsequence:
for any $\varepsilon>0$, there exists $T(\phi,\varepsilon)$ such that for all $T\ge T(\phi,\varepsilon)$ there exists $R(\phi,\varepsilon,T)$ such that for all $R\ge R(\phi,\varepsilon,T)$, there holds that
\begin{equation}\label{ProxyEuclHyp}
\Vert U_k-\tilde{u}_k\Vert_{X^1(\{\vert t-t_k\vert\le TN_k^{-2}\}\cap\{\vert t\vert\le T^{-1}\})}\le\varepsilon,
\end{equation}
for $k$ large enough, where
\begin{equation*}
\tilde{u}_k(x,t)=N_k^\frac{1}{2}\eta(N_kd_g(x_k,x)/R)u(N_k\exp_{x_k}^{-1}(x),N_k^2(t-t_k)).
\end{equation*}
In addition, up to a subsequence\footnote{The definition of $T_N$ is given in \eqref{rescaled}.},
\begin{equation}\label{ScatEuclSol}
\Vert U_k(t)-\Pi_{t_k-t,x_k}T_{N_k}\phi^{\pm\infty}\Vert_{X^1(\{\pm(t-t_k)\geq TN_k^{-2}\}\cap \{\vert t \vert\le T^{-1}\})}\le \varepsilon,
\end{equation}
for $k$ large enough (depending on $\phi,\varepsilon,T,R$).
\end{proposition}

\begin{proof}[Proof of Proposition \ref{GEForEP}]

This is follows from minor adaptation of the proof in \cite[Proposition 4.4]{IoPa2}. Here Lemma \ref{Extinction} is used in an essential way.
\end{proof}

\subsection{Profile decomposition}

In this section we show that given a bounded sequence of functions $f_k\in H^1(\mathbb{S}^3)$ we can construct suitable {\it{profiles}} and express the sequence in terms of these profiles. The statements and the arguments in this section are very similar to those in \cite[Section 5]{IoPa2}. See also \cite{IoPa,IoPaSt} and \cite{Ker} for the original proofs of Keraani in the Euclidean geometry and \cite{BaGe,MeVe} for earlier results.

The following is our main definition.

\begin{definition}\label{DefPro}

\begin{enumerate}

\item We define a Euclidean frame to be a sequence $\mathcal{F}_e=(N_k,t_k,x_k)_k$ with $N_k\ge 1$, $N_k\to+\infty$, $t_k\in\mathbb{R}$, $t_k\to 0$, $x_k\in\mathbb{S}^3$. We say that two frames $(N_k,t_k,x_k)_k$ and $(M_k,s_k,y_k)_k$ are orthogonal if
\begin{equation*}
\lim_{k\to+\infty} \left(\left\vert \ln\frac{N_k}{M_k}\right\vert+N_k^2\vert t_k-s_k\vert+N_kd_g(x_k,y_k)\right)=+\infty.\end{equation*}
Two frames that are not orthogonal are called equivalent.

\item If $\mathcal{O}=(N_k,t_k,x_k)_k$ is a Euclidean frame and if $\phi\in \dot{H}^1(\mathbb{R}^3)$, we define the Euclidean profile associated to $(\phi,\mathcal{O})$ as the sequence $\widetilde{\phi}_{\mathcal{O}_k}$
\begin{equation*}
\widetilde{\phi}_{\mathcal{O}_k}(x):=\Pi_{t_k,x_k}(T_{N_k}\phi).
\end{equation*}
\end{enumerate}
\end{definition}

The following lemma summarizes some of the basic properties of profiles associated to equivalent/orthogonal frames. Its proof uses Lemma \ref{step1} with $\rho=0$ to control linear evolutions inside the Euclidean window and Lemma \ref{Extinction} to control these evolutions outside such a window. Given these ingredients, the proof of Lemma \ref{EquivFrames} is very similar to the proof of Lemma 5.4 in \cite{IoPaSt}, and is omitted.

\begin{lemma}(Equivalence of frames)\label{EquivFrames}

(i) If $\mathcal{O}$ and $\mathcal{O}^\prime$ are equivalent Euclidean profiles, then there exists an isometry $T:\dot{H}^1(\mathbb{R}^3)\to\dot{H}^1(\mathbb{R}^3)$ such that for any profile $\widetilde{\psi}_{\mathcal{O}^\prime_k}$, up to a subsequence there holds that
\begin{equation}\label{equiv}
\limsup_{k\to+\infty}
\Vert \widetilde{T\psi}_{\mathcal{O}_k}-\widetilde{\psi}_{\mathcal{O}^\prime_k}\Vert_{H^1(\mathbb{S}^3)}=0.
\end{equation}

(ii) If $\mathcal{O}$ and $\mathcal{O}^\prime$ are orthogonal frames and $\widetilde{\psi}_{\mathcal{O}_k}$, $\widetilde{\varphi}_{\mathcal{O}^\prime_k}$ are corresponding profiles, then, up to a subsequence,
\begin{equation*}
\begin{split}
\lim_{k\to+\infty}\langle \widetilde{\psi}_{\mathcal{O}_k},\widetilde{\varphi}_{\mathcal{O}^\prime_k}\rangle_{H^1\times H^1(\mathbb{S}^3)}&=0,\\
\lim_{k\to+\infty}\langle |\widetilde{\psi}_{\mathcal{O}_k}|^3,|\widetilde{\varphi}_{\mathcal{O}^\prime_k}|^3\rangle_{L^2\times L^2(\mathbb{S}^3)}&=0.
\end{split}
\end{equation*}

(iii) If $\mathcal{O}$ is a Euclidean frame and $\widetilde{\psi}_{\mathcal{O}_k}$, $\widetilde{\varphi}_{\mathcal{O}_k}$ are two profiles corresponding to $\mathcal{O}$, then
\begin{equation*}
\begin{split}
&\lim_{k\to+\infty}\left(\Vert\widetilde{\psi}_{\mathcal{O}_k}\Vert_{L^2(\mathbb{S}^3)}+\Vert\widetilde{\varphi}_{\mathcal{O}_k}\Vert_{L^2(\mathbb{S}^3)}\right)=0,\\
&\lim_{k\to+\infty}\langle \widetilde{\psi}_{\mathcal{O}_k},\widetilde{\varphi}_{\mathcal{O}_k}\rangle_{H^1\times H^1(\mathbb{S}^3)}=\langle \psi,\varphi\rangle_{\dot{H}^1\times\dot{H}^1(\mathbb{R}^3)}.
\end{split}
\end{equation*}
\end{lemma}

\begin{definition}\label{absent}
We say that a sequence of functions $\{f_k\}_k\subseteq H^1(\mathbb{S}^3)$ is absent from a frame $\mathcal{O}$ if for every profile $\psi_{\mathcal{O}_k}$ associated to $\mathcal{O}$,
\begin{equation*}
\int_{\mathbb{S}^3}\left(f_k\overline{\widetilde{\psi}}_{\mathcal{O}_k}+\nabla f_k\nabla\overline{\widetilde{\psi}}_{\mathcal{O}_k}\right)dx\to0
\end{equation*}
as $k\to+\infty$.
\end{definition}

Note in particular that a profile associated to a frame $\mathcal{O}$ is absent from any frame orthogonal to $\mathcal{O}$.

The following proposition is the core of this section. Its proof is similar to the proof of \cite[Proposition 5.5]{IoPa}, and is omitted.

\begin{proposition}\label{PD}
Consider $\{f_k\}_k$ a sequence of functions in $H^1(\mathbb{S}^3)$ satisfying
\begin{equation}\label{FkBoundedPD}
\limsup_{k\to+\infty}\Vert f_k\Vert_{H^1(\mathbb{S}^3)}\lesssim E
\end{equation}
and a sequence of intervals $I_k=(-T_k,T^k)$ such that $\vert I_k\vert\to0$ as $k\to+\infty$.
Up to passing to a subsequence, assume that $f_k\rightharpoonup g\in H^1(\mathbb{S}^3)$.
There exists a sequence of profiles $\widetilde{\psi}^\alpha_{\mathcal{O}^\alpha_k}$ associated to pairwise orthogonal Euclidean frames $\mathcal{O}^\alpha$ such that, after extracting a subsequence, for every $J\ge 0$
\begin{equation}\label{DecompositionPD}
f_k=g+\sum_{1\le \alpha\le J}\widetilde{\psi}^\alpha_{\mathcal{O}^\alpha_k}+R_k^J
\end{equation}
where $R_k^J$ is absent from the frames $\mathcal{O}^\alpha$, $\alpha\le J$ and is small in the sense that
\begin{equation}\label{smallnessPD}
\limsup_{J\to+\infty}\limsup_{k\to+\infty}\big[\sup_{N\ge 1,t\in I_k,\,x\in\mathbb{S}^3}N^{-\frac{1}{2}}\left\vert \left(e^{itL}P_NR_k^J\right)(x)\right\vert\big]=0.
\end{equation}
Besides, we also have the following orthogonality relations
\begin{equation}\label{OrthogonalityPD}
\begin{split}
&\Vert f_k\Vert_{L^2}^2=\Vert g\Vert_{L^2}^2+\Vert R_k^J\Vert_{L^2}^2+o_k(1),\\
&\Vert \nabla f_k\Vert_{L^2}^2=\Vert \nabla g\Vert_{L^2}^2+\sum_{\alpha\le J}\Vert\nabla_{\mathbb{R}^3}\psi^\alpha\Vert_{L^2(\mathbb{R}^3)}^2+\Vert\nabla R_k^J\Vert_{L^2}^2+o_k(1),\\
&\lim_{J\to+\infty}\limsup_{k\to+\infty}\left\vert\Vert f_k\Vert_{L^6}^6-\Vert g\Vert_{L^6}^6-\sum_{\alpha\le J}\Vert\widetilde{\varphi}^\alpha_{\mathcal{O}^\alpha_k}\Vert_{L^6}^6\right\vert=0,
\end{split}
\end{equation}
where $o_k(1)\to0$ as $k\to+\infty$, possibly depending on $J$.
\end{proposition}

The proof of the last bound in \eqref{OrthogonalityPD} relies on the estimate
\begin{equation*}
 \limsup_{J\to+\infty}\limsup_{k\to+\infty}\|R_k^J\|_{L^6(\mathbb{S}^3)}=0.
\end{equation*}
This is a consequence of \eqref{smallnessPD} and the bound
\begin{equation}\label{Sobo}
 \|f\|^6_{L^6(\mathbb{S}^3)}\lesssim\|f\|_{H^1(\mathbb{S}^3)}^2\big(\sup_{N\geq 1}N^{-1/2}\|P_Nf\|_{L^\infty(\mathbb{S}^3)}\big)^4,
\end{equation}
for any $f\in H^1(\mathbb{S}^3)$, see for example \cite[Lemma 2.3]{IoPa} for a similar proof.


\section{Global Existence}\label{SecGWP}

\subsection{Induction on Energy}

We follow a strategy derived from \cite{KeMe}. From Proposition \ref{LWP}, we see that to prove Theorem \ref{MainThm}, it suffices to prove that solutions remain bounded in $Z$ on intervals of length at most $1$. To obtain this, we induct on the energy $E(u)$.

Define
\begin{equation*}
\Lambda_\ast(L)=\limsup_{\tau\to0}\sup\{\Vert u\Vert_{Z(I)}^2,E(u)\le L,\vert I\vert\le \tau\}
\end{equation*}
where the supremum is taken over all strong solutions of
\eqref{NLSBis} of energy less than or equal to $L$ and all intervals $I$ of length $\vert I\vert\le \tau$. In addition, define
\begin{equation}\label{Emax}
E_{max}=\sup\{L: \Lambda_\ast(L)<+\infty\}.
\end{equation}
We see that Theorem \ref{MainThm} is equivalent to the following statement.

\begin{theorem}\label{Alex40}
$E_{max}=+\infty$. In particular every solution of \eqref{NLSBis} is global.
\end{theorem}

\begin{proof}[Proof of Theorem \ref{Alex40}] Suppose for contradiction that $E_{max}<+\infty$. From now on, all our constants are allowed to depend on $E_{max}$. By definition, there exists a sequence of intervals $I_k$ and a sequence of solutions $u_k$ such that
\begin{equation}\label{CondForComp}
E(u_k)\to E_{max},\quad\vert I_k\vert\to0,\quad \Vert u_k\Vert_{Z(I_k)}\to+\infty
\end{equation}
and $0\in I_k$.
We now apply Proposition \ref{PD} to the sequence $\{u_k(0)\}_k$ with $I_k$. This gives a sequence of profiles $\widetilde{\psi}^\alpha_{\mathcal{O}^\alpha_k}$, $\alpha,k=1,2,\dots$, and a decomposition
\begin{equation*}
u_k(0)=g+\sum_{1\le\alpha\le J}\widetilde{\psi}^\alpha_{\mathcal{O}^\alpha_k}+R^J_k.
\end{equation*}

Using Lemma \ref{EquivFrames} and passing to a subsequence, we may renormalize every Euclidean profile, that is, up to passing to an equivalent profile, 
we may assume that for every Euclidean frame $\mathcal{O}^\alpha$, $\mathcal{O}^\alpha\in\widetilde{\mathcal{F}}_e$, see definition \eqref{renframe}.
Besides, using Lemma \ref{EquivFrames} and passing to a subsequence once again, we may assume that for every $\alpha\ne\beta$,
either $N^\alpha_k/N^\beta_k+N^\beta_k/N^\alpha_k\to+\infty$ as $k\to+\infty$ or $N^\alpha_k=N^\beta_k$ for all $k$ and in this case, either $t^\alpha_k=t^\beta_k$ as $k\to+\infty$ or $(N^\alpha_k)^2\vert t^\alpha_k-t^\beta_k\vert \to+\infty$ as $k\to+\infty$.

From \eqref{OrthogonalityPD} and Lemma \ref{EquivFrames} (iii) we see that, after extracting a subsequence,
\begin{equation}\label{SumOfL}
\begin{split}
&E(\alpha):=\lim_{k\to+\infty}E(\widetilde{\psi}^\alpha_{\mathcal{O}^\alpha_k})\in(0,E_{max}],\\
&\lim_{J\to+\infty}\big[\sum_{1\le\alpha\le J}E(\alpha)+\lim_{k\to+\infty}E(R_k^J)\big]\le E_{max}-E(g).
\end{split}
\end{equation}

We consider also the remainder and note that, for $p\in\{p_0,p_1\}$ and $q=(p_0+4)/2>4$,
\begin{equation*}
\begin{split}
\sum_N N^{5-p/2}&\Vert P_Ne^{itL}R^J_k\Vert_{L^{p}_{x,t}(\mathbb{S}^3\times I_k)}^{p}\\
&\lesssim \big[\sup_NN^{-\frac{1}{2}}\Vert e^{itL}P_N R^J_k\Vert_{L^\infty_{x,t}(\mathbb{S}^3\times I_k)}\big]^{p-q}\sum_N \left[N^{5/q-1/2}\Vert P_Ne^{itL}R^J_k\Vert_{L^q_{x,t}(\mathbb{S}^3\times I_k)}\right]^q\\
&\lesssim \big[\sup_NN^{-\frac{1}{2}}\Vert e^{itL}P_N R^J_k\Vert_{L^\infty_{x,t}(\mathbb{S}^3\times I_k)}\big]^{p-q}\sum_N N^q\Vert P_N R^J_k\Vert_{L^2_x(\mathbb{S}^3)}^q\\
&\lesssim\big[\sup_NN^{-\frac{1}{2}}\Vert e^{itL}P_N R^J_k\Vert_{L^\infty_{x,t}(\mathbb{S}^3\times I_k)}\big]^{p-q}.
\end{split}
\end{equation*}
Therefore
\begin{equation}\label{SmallnessRterm}
\limsup_{J\to+\infty}\limsup_{k\to+\infty}\Vert e^{itL}R^J_k\Vert_{Z(I_k)}=0.
\end{equation}

\medskip

{\bf Case I:} $\{u_k(0)\}_k$ converges strongly in $H^1(\mathbb{S}^3)$ to its limit $g$ which satisfies $E(g)=E_{max}$. Then, by Strichartz estimates, there exists $\eta>0$ such that, for $k$ large enough
\begin{equation*}
\Vert e^{itL}u_k(0)\Vert_{Z(I_k)}\le\Vert e^{itL}g\Vert_{Z(-\eta,\eta)}+o_k(1)\le\delta_0,
\end{equation*}
where $\delta_0$ is given by the local theory in Proposition \ref{LWP}. In this case, we conclude that $\Vert u_k\Vert_{Z(I_k)}\lesssim 2\delta_0$ which contradicts \eqref{CondForComp}.

\medskip

{\bf Case IIa:} $g=0$ and there are no profiles. Then, taking $J$ sufficiently large, we get that, for $k$ large enough,
\begin{equation*}
\Vert e^{itL}u_k(0)\Vert_{Z(I_k)}\le\delta_0,
\end{equation*}
where $\delta_0$ is as above. Once again, this contradicts \eqref{CondForComp}.

\medskip

{\bf Case IIb:} $g=0$ and there is only one Euclidean profile, such that
\begin{equation*}
u_k(0)=\widetilde{\psi}_{\mathcal{O}_k}+o_k(1)
\end{equation*}
in $H^1(\mathbb{S}^3)$ (see \eqref{SumOfL}), where $\mathcal{O}$ is a Euclidean frame. In this case, we let $U_k$ be the solution of \eqref{NLSBis} with initial data $U_k(0)=\widetilde{\psi}_{\mathcal{O}_k}$ and we use \eqref{ControlOnZNormForEP} to get, for $k$ large enough
\begin{equation*}
\Vert U_k\Vert_{Z(I_k)}\le\Vert U_k\Vert_{Z(-\delta,\delta)}\lesssim 1\quad\text{and}\quad\lim_{k\to +\infty}\Vert U_k(0)-u_k(0)\Vert_{H^1}\to 0.
\end{equation*}
We may use Proposition \ref{Stabprop} to deduce that
\begin{equation*}
\Vert u_k\Vert_{Z(I_k)}\lesssim \Vert u_k\Vert_{X^1(I_k)}\lesssim 1
\end{equation*}
which contradicts \eqref{CondForComp}.

\medskip

{\bf Case III:} $E(g)<E_{max}$ and $E(\alpha)<E_{max}$ for any $\alpha=1,2,\dots$. 
Up to relabeling the profiles, we can assume that for all $\alpha$, $E(\alpha)\le E(1)<E_{max}-\eta$, $E(g)<E_{max}-\eta$ for some $\eta>0$. 
Now for every linear profile $\widetilde{\psi}^\alpha_{\mathcal{O}^\alpha_k}$, we define the associated nonlinear profile $U^\alpha_k$ as the maximal solution of \eqref{NLSBis} with initial data $U^\alpha_k(0)=\widetilde{\psi}^\alpha_{\mathcal{O}^\alpha_k}$. 
A more precise description of each nonlinear profile is given by Proposition \ref{GEForEP}. Similarly, we define $W$ to be the nonlinear solution of \eqref{NLSBis} with initial data $g$. In view of the induction hypothesis
\begin{equation*}
\Vert W\Vert_{Z(-1,1)}+\Vert U^\alpha_k\Vert_{Z(-1,1)}\le 3\Lambda(E_{max}-\eta/2,2)\lesssim 1,
\end{equation*}
where from now on all the implicit constants are allowed to depend on $\Lambda(E_{max}-\eta/2,2)$. Using Proposition \ref{Stabprop} it follows that for any $\alpha$ and any $k>k_0(\alpha)$ sufficiently large,
\begin{equation}\label{BddX1}
\Vert W\Vert_{X^1(-1,1)}+\Vert U^\alpha_k\Vert_{X^1(-1,1)}\lesssim 1.
\end{equation}

For $J,k\geq 1$ we define
\begin{equation*}
U^J_{prof,k}:=W+\sum_{\alpha=1}^J U^\alpha_k.
\end{equation*}
We show first that there is a constant $Q$ such that
\begin{equation}\label{bi1}
\Vert U^J_{prof,k}\Vert_{X^1(-1,1)}^2+\Vert W\Vert_{X^1(-1,1)}^2+\sum_{\alpha=1}^J\Vert U^\alpha_k\Vert_{X^1(-1,1)}^2+\sum_{\alpha=1}^J\Vert U^\alpha_k-e^{itL}\widetilde{\psi}^\alpha_{\mathcal{O}^\alpha_k}\Vert_{X^1(-1,1)}\le Q^2,
\end{equation}
uniformly in $J$, for all $k\ge k_0(J)$ sufficiently large. Indeed, a simple fixed point argument as in Section \ref{localwp} shows that there exists $\delta_0>0$ such that if
\begin{equation*}
\Vert \phi\Vert_{H^1(\mathbb{S}^3)}=\delta\le\delta_0
\end{equation*}
then the unique strong solution of \eqref{NLSBis} with initial data $\phi$ is global and satisfies
\begin{equation}\label{SmalldataCCL}
\begin{split}
\Vert u\Vert_{X^1(-2,2)}&\le 2\delta\quad\text{and}\quad\Vert u-e^{itL}\phi\Vert_{X^1(-2,2)}\lesssim \delta^2.
\end{split}
\end{equation}
From \eqref{SumOfL}, we know that there are only finitely many profiles such that $E(\alpha)\ge\delta_0/2$. Without loss of generality, we may assume that for all $\alpha\ge A$, $E(\alpha)\le\delta_0$. Using \eqref{OrthogonalityPD}, \eqref{BddX1}, and \eqref{SmalldataCCL} we then see that
\begin{equation*}
\begin{split}
&\Vert U^J_{prof,k}\Vert_{X^1(-1,1)}=\Vert  W+\sum_{1\le\alpha\le J}U^\alpha_k\Vert_{X^1(-1,1)}\\
&\le \Vert W\Vert_{X^1(-1,1)}+\sum_{1\le\alpha\le A}\Vert U^\alpha_k\Vert_{X^1(-1,1)}+\Vert \sum_{A\le\alpha\le J}(U^\alpha_k-e^{itL}U^\alpha_k(0))\Vert_{X^1(-1,1)}\\
&+\Vert e^{itL}\sum_{A\le\alpha\le J}U^\alpha_k(0)\Vert_{X^1(-1,1)}\\
&\lesssim 1+A+\sum_{A\le\alpha\le J}E(\alpha)+\Vert\sum_{A\le\alpha\le J}U^\alpha_k(0)\Vert_{H^1}\lesssim 1.
\end{split}
\end{equation*}
The bound on $\sum_{\alpha=1}^J\Vert U^\alpha_k\Vert_{X^1(-1,1)}^2$ is similar (in fact easier), which gives \eqref{bi1}.

We now claim that
\begin{equation*}
U^J_{app,k}=W+\sum_{1\le\alpha\le J}U^\alpha_k+e^{itL}R^J_k
\end{equation*}
is an approximate solution for all $J\ge J_0$ and all $k\ge k_0(J)$ sufficiently large. We saw in \eqref{bi1} that $U^J_{app,k}$ has bounded $X^1$-norm. Let $\varepsilon=\varepsilon (2Q^2)$ be the constant given in Proposition \ref{Stabprop}. 
We compute, with $F(z)=z|z|^4$,
\begin{equation*}
\begin{split}
e&=\left(i\partial_t+L\right)U^J_{app,k}-F(U^J_{app,k})=F(U^J_{app,k})-F(W)-\sum_{1\le\alpha\le J}F(U^\alpha_k)\\
&=F(U^J_{prof,k}+e^{itL}R^J_k)-F(U^J_{prof,k})+F(U^J_{prof,k})-F(W)-\sum_{1\le\alpha\le J}F(U^\alpha_k).
\end{split}
\end{equation*}
and appealing to Lemma \ref{AlmostSol} below, we obtain that
\begin{equation*}
\limsup_{k\to+\infty}\Vert e\Vert_{N(I_k)}\le\varepsilon/2
\end{equation*}
for $J\ge J_0(\varepsilon)$. In this case, we may use Proposition \ref{Stabprop} to conclude that $u_k$ satisfies
\begin{equation*}
\Vert u_k\Vert_{X^1(I_k)}\lesssim \Vert U^J_{app,k}\Vert_{X^1(I_k)}\le \Vert U^J_{prof,k}\Vert_{X^1(-1,1)}+\Vert e^{itL}R^J_k\Vert_{X^1(-1,1)}\lesssim 1,
\end{equation*}
which contradicts \eqref{CondForComp}. This finishes the proof.
\end{proof}

We have now proved our main theorem, except for the following important assertion.

\begin{lemma}\label{AlmostSol}

With the notations in {\bf Case III} of the proof of Theorem \ref{Alex40}, we have that, for fixed $J$,
\begin{equation}\label{Asol1}
\limsup_{k\to+\infty}\Vert F(U^J_{prof,k})-F(W)-\sum_{1\le\alpha\le J}F(U^\alpha_k)\Vert_{N(I_k)}=0.
\end{equation}
Besides, we also have that
\begin{equation}\label{Asol2}
\limsup_{J\to+\infty}\limsup_{k\to+\infty}\Vert F(U^J_{prof,k}+e^{itL}R^J_k)-F(U^J_{prof,k})\Vert_{N(I_k)}=0.
\end{equation}
\end{lemma}

The proof of this Lemma is identical to the proof in \cite[Section 7]{IoPa2}, with \cite[Lemma $7.1$]{IoPa2} replaced by Lemma \ref{HFLFI} below.

Recall from Section \ref{SecNot} that $\mathfrak{O}_{4,1}(a,b)$ denotes a quantity which is quartic in $\{a,\overline{a}\}$ and linear in $\{b,\overline{b}\}$.

\begin{lemma}\label{HFLFI}
Let $O\in\mathbb{S}^3$ and assume that $B,N\geq 2$ are dyadic numbers and $\omega:\mathbb{S}^3\times(-1,1)\to\mathbb{C}$ is a function satisfying $\vert \nabla^j\omega\vert \leq N^{j+1/2}\mathbf{1}_{\{d_g(x,O)\leq N^{-1},\,|t|\leq N^{-2}\}}$, $j=0,1$. Then
\begin{equation*}
\Vert \mathfrak{O}_{4,1}(\omega,e^{itL}P_{> BN}f)\Vert_{L^1((-1,1),H^1)}\lesssim B^{-1}\|f\|_{H^1(\mathbb{S}^3)}.
\end{equation*}
\end{lemma}

\begin{proof}[Proof of Lemma \ref{HFLFI}] The general strategy of the proof is similar to the one in \cite{IoPa2} on $\mathbb{T}^3$.
We may assume that $\|f\|_{H^1(\mathbb{S}^3)}=1$ and $f=P_{>BN}f$. We notice that
\begin{equation*}
\begin{split}
\Vert \mathfrak{O}_{4,1}(\omega,e^{itL}P_{>BN}f)\Vert_{L^1((-1,1),H^1)}&\lesssim \Vert \mathfrak{O}_{4,1}(\omega,\nabla e^{itL}f)\Vert_{L^1((-1,1),L^2)}\\
&+\Vert e^{itL}f\Vert_{L^\infty_tL^2_x}\Vert \omega\Vert_{L^4_tL^\infty_x}^3\Vert \vert\nabla\omega\vert+\vert\omega\vert\Vert_{L^4_tL^\infty_x}\\
&\lesssim \Vert \mathfrak{O}_{4,1}(\omega,\nabla e^{itL}f)\Vert_{L^1((-1,1),L^2)}+B^{-1}.
\end{split}
\end{equation*}
Let $\chi_N=\mathfrak{1}_{B(O,2N^{-1})}$ and $W(x,t):=N^4\chi_N(x)\eta(N^2t)$ and write
\begin{equation*}
\begin{split}
\Vert \mathfrak{O}_{4,1}(\omega,\nabla e^{itL}f)\Vert_{L^1((-1,1),L^2)}^2&\lesssim N^{-2}\Vert W^\frac{1}{2}\nabla e^{itL}f\Vert_{L^2(\mathbb{S}^3\times (-1,1))}^2\\
&\lesssim N^{-2}\sum_{j=1}^3\int_{-1}^1\langle e^{itL}\partial_jf,We^{itL}\partial_jf\rangle_{L^2\times L^2(\mathbb{S}^3)} dt\\
&\lesssim N^{-2}\sum_{j=1}^3\langle \partial_jf,\left[\int_{-1}^1e^{-itL}We^{itL}dt\right] \partial_jf\rangle_{L^2\times L^2(\mathbb{S}^3)}.
\end{split}
\end{equation*}
Therefore, it remains to prove that
\begin{equation}\label{Alex50}
\|K\|_{L^2(\mathbb{S}^3)\to L^2(\mathbb{S}^3)}\lesssim N^2B^{-1}\,\,\text{ where }\,\,K=P_{>BN}\int_{\mathbb{R}}e^{-itL}We^{itL}P_{>BN}\,dt.
\end{equation}

We look at the Fourier coefficients

\begin{equation*}
\begin{split}
K_{p,q}&=\pi_pK\pi_q\\
&=N^4(1-\eta(p/BN))(1-\eta(q/BN))\int_{\mathbb{R}}e^{-it\left[p^2-q^2\right]}\eta(N^2t)dt\cdot\left[\pi_p\chi_N\pi_q\right]\\
&=N^2(1-\eta(p/BN))(1-\eta(q/BN))\hat{\eta}(N^{-2}(p^2-q^2))\cdot\left[\pi_p\chi_N\pi_q\right].
\end{split}
\end{equation*}
Using Schur's lemma, it suffices to prove that
\begin{equation*}
\sup_{p\ge BN}\sum_{q\in\mathbb{Z}}(1-\eta(q/BN))\vert \hat{\eta}(N^{-2}(p^2-q^2))\vert\Vert \pi_p\chi_N\pi_q\Vert_{L^2\to L^2}\lesssim B^{-1}.
\end{equation*}
The new ingredient we need is the following
\begin{equation}\label{CrucialGain}
\Vert \pi_p\chi_N\pi_q\Vert_{L^2\to L^2}\lesssim N^{-1}+\min(p,q)^{-2}
\end{equation}
which is a consequence\footnote{Note that we use both bounds in \eqref{CrucialGain1}.} of \eqref{CrucialGain1}.
Assuming \eqref{CrucialGain}, we finish the proof as follows: for any $p\ge BN$,
\begin{equation*}
\begin{split}
\sum_{q\in\mathbb{Z}}(1-\eta(q/BN))\vert \hat{\eta}(N^{-2}(p^2-q^2))\vert\Vert \pi_p\chi_N\pi_q\Vert_{L^2\to L^2}&\lesssim \sum_{q\ge BN}N^{-1}\cdot\left[1+N^{-2}\vert p^2-q^2\vert\right]^{-10}\\
&\lesssim \sum_{q\ge BN}N^{-1}\cdot\left[1+B\vert p-q\vert/N\right]^{-10}\\
&\lesssim B^{-1}
\end{split}
\end{equation*}
which finishes the proof.
\end{proof}

\section{Appendix}\label{App}

\subsection{Weyl Sum estimate}
For a sequence $c=(c_p)_p$, we define the linear difference operator $\delta$ by
\begin{equation*}
(\delta c)_p=c_{p}-c_{p-1}
\end{equation*}
and for $j\ge 1$, $\delta^{j+1}c=\delta(\delta^jc)$. The following lemma is essentially from \cite{Bo2} in a slightly different formulation. 
\begin{lemma}\label{BouLem}
Assume that $(c_p)_p$ satisfies
\begin{equation*}
\vert\delta^j c\vert\lesssim KN^{-j},\quad 0\le j\le 2
\end{equation*}
and that
\begin{equation*}
\{p:c_p\ne 0\}\subset [-QN,QN]
\end{equation*}
For $t\in [-\pi,\pi]$ let $t/\pi=a/q+\beta$, $0\le \vert a\vert\le q\le N$ and $\vert\beta\vert\le1/(Nq)$ be its Dirichlet approximation.
Define
\begin{equation*}
S(t)=\sum_pc_pe^{it\vert p\vert^{2}},
\end{equation*}
then there holds that
\begin{equation}\label{LinftyBound}
\vert S(t)\vert\lesssim KQ^\frac{3}{2}\frac{N}{\sqrt{q(1+N^{2}\vert\beta\vert)}}.
\end{equation}
\end{lemma}

\begin{proof}
We may assume that $K=1$.
We first compute
\begin{equation*}
\begin{split}
\vert S\vert^2&=\sum_{a,b}\overline{c}_ac_be^{it\left[\vert b\vert^{2}-\vert a\vert^{2}\right]}=\sum_{m}e^{i t\vert m\vert^{2}}\sigma_m\\
\sigma_m&=\sum_{p}\overline{c}_{p}c_{p+m}e^{it 2mp}.
\end{split}
\end{equation*}
We shall not use the oscillations that might be present in the above sum beyond the following claim:
\begin{equation}\label{Claim1}
\vert\sigma_m\vert\lesssim\frac{NQ}{\left[1+N\hbox{dist}(mt/\pi,\mathbb{Z})\right]^2}.
\end{equation}
If $N\hbox{dist}(mt/\pi,\mathbb{Z})<1$, the bound is clear. Otherwise, we simply observe that, letting $z=e^{i2mt}$ and $C_p=\overline{c}_pc_{p+m}$, there holds, uniformly in $m$,
\begin{equation*}
\begin{split}
(1-z)\sum_pC_pz^p&=\sum_p(\delta C)_pz^p\\
(1-z)^{2}\sum_pC_pz^p&=\sum_p(\delta^{2} C)_pz^p\\
\vert \delta^jC_p\vert&\lesssim N^{-j},\quad 0\le j\le 2.
\end{split}
\end{equation*}
This gives \eqref{Claim1}.

\medskip

Now, we can finish the proof. We may assume that $a\ge 0$ and $\vert\beta\vert\ne 0$. For any $m\in\mathbb{Z}$, we define
\begin{equation*}
b(m)=am \mod q,\quad b(m)\in\mathbb{Z}_q=\{0,1,\dots,q-1\}.
\end{equation*}
Since $(a,q)=1$, $a$ is invertible in $\mathbb{Z}_q$ and the mapping $r\mapsto b(r)$ is a bijection $\mathbb{Z}_q\to\mathbb{Z}_q$. We now distinguish two cases.

\medskip

The nonresonant case\footnote{This case is of course vacuous if $q\le 10Q$.}: $b(r)\notin \mathcal{R}=\{0,1,.,3Q,q-3Q,\dots,q-2,q-1\}$. In this case, since $\vert m\vert\le 2QN$ and
\begin{equation*}
mt/\pi=\frac{ma}{q}+m\beta\in \mathbb{Z}+\frac{b(m)}{q}+[-\frac{2Q}{q},\frac{2Q}{q}],
\end{equation*}
we may use the oscillations in $b(m)$ since
\begin{equation*}
\hbox{dist}(mt/\pi,\mathbb{Z})=\frac{b(m)}{q}+m\beta\ge\frac{3}{5}\min\left\{\frac{b(m)}{q},\frac{q-b(m)}{q}\right\}
\end{equation*}
so that, we can estimate the corresponding contribution by
\begin{equation*}
\begin{split}
\sum_{m:b(m)\notin\mathcal{R}}\left\vert\sigma_m\right\vert\lesssim\frac{Q}{N}\sum_{m:b(m)\notin\mathcal{R}}\frac{q^2}{[b(m)]^2}\lesssim\frac{Qq^2}{N}\sum_{k\ge 2}\sum_{m:b(m)=k}\frac{1}{k^2}\lesssim \frac{Q^2q^2}{N}\frac{N}{q}\lesssim Q^2q
\end{split}
\end{equation*}
which is acceptable.

\medskip

The resonant case. In this case, we are left with a worse bound in \eqref{Claim1}, but fortunately, there are only $6Q$ of them and we can estimate them one by one. Thus, from now on, we assume that $b(m)$ is fixed. Then, clearly,
\begin{equation*}
\{\hbox{dist}(mt/\pi,\mathbb{Z})\,\,\,:\,\,\,b(m)=k\}
\end{equation*}
is contained in at most $1+Q/q$ arithmetic sequences of length $O(N)$ and increment $2q\vert \beta\vert$. Hence its contribution can be estimated by
\begin{equation*}
\begin{split}
Q\min(\frac{N}{q}QN,\sum_{k\ge 0}\frac{NQ}{(1+N2kq\vert\beta\vert)^2})&\lesssim Q^2\min(\frac{N^2}{q},\sum_{2k q\vert\beta\vert N\le 1}N+\sum_{2k q\vert\beta\vert N\ge 1}\frac{1}{N(kq\vert\beta\vert)^2})\\
&\lesssim Q^2\min(\frac{N^2}{q},\frac{1}{q\vert\beta\vert}).
\end{split}
\end{equation*}
Again, this is acceptable.
\end{proof}

\subsection{The case of the ball with Dirichlet boundary conditions and radial data}\label{SecBall}

Here we give the main ingredients to prove Theorem \ref{BallThm}. The analysis of the Dirichlet problem on $B(0,\pi)$ is not so different from the analysis on $\mathbb{S}^3$ due to the relation
\begin{equation*}
(1-\Delta_{\mathbb{S}^3})f=\frac{\theta^2}{\sin^2\theta}\Delta_{\mathbb{R}^3}\left[\frac{\sin^2\theta}{\theta^2}f\right]
\end{equation*}
where $\theta$ denotes the distance to the origin\footnote{Here we identify functions on $\mathbb{S}^3$ with functions on $B(0,\pi)$ through the relation $f(x)\simeq f(\exp_Ox)$, where $O$ denotes the north pole.}. This leads to the relation
\begin{equation}\label{EquivFlow}
\left(e^{-it\Delta_{B^3_D}}\varphi\right)=g\cdot e^{itL}\left(\frac{\varphi}{g}\right),\quad g(\theta,\omega)=\frac{\sin(\theta)}{\theta}.
\end{equation}
Since we also have that
\begin{equation*}
\Vert \frac{\varphi}{g}\Vert_{L^2(\mathbb{S}^3)}=\Vert \varphi\Vert_{L^2(B^3)},
\end{equation*}
we can directly transfer the linear estimates on $\mathbb{S}^3$ to estimates on the ball with Dirichlet condition. In particular, we recover all the results of Section \ref{localwp}. In Section \ref{SecProf}, we also see that Lemma \ref{Extinction} holds directly, while the other lemmas do not depend on the geometry and hence trivially hold. Note in particular that the radial Sobolev inequality
\begin{equation*}
\left(\frac{\vert x\vert}{\pi-\vert x\vert}\right)^\frac{1}{2}\vert u(x)\vert\lesssim\Vert\nabla u\Vert_{L^2(B)},
\end{equation*}
valid for all functions vanishing at $\pi$ forces all the Euclidean profiles to only concentrate at the origin. In Section \ref{SecGWP}, the main novelty is in the linear Lemma \ref{HFLFI}, which again holds equally, thanks to \eqref{EquivFlow}. The other parts of the proof need only minor modifications.

\end{document}